\documentclass[12pt]{amsart}

\usepackage{graphicx}
\usepackage{amssymb}
\usepackage{float}
\usepackage{fullpage}

\newtheorem{theorem}{Theorem}[section]
\newtheorem{lemma}[theorem]{Lemma}
\newtheorem{corollary}[theorem]{Corollary}
\newtheorem{proposition}[theorem]{Proposition}

\theoremstyle{definition}
\newtheorem{definition}[theorem]{Definition}

\theoremstyle{remark}
\newtheorem{remark}[theorem]{Remark}

\numberwithin{equation}{section}

\usepackage{physics}
\usepackage{multicol}

\usepackage{comment}

\usepackage{mathtools}

\newcommand{\R}{{\mathbb R}}

\newcommand{\Z}{{\mathbb Z}}
\newcommand{\N}{{\mathbb N}}

\newcommand{\D}{{\mathbb D}}

\renewcommand{\phi}{{\varphi}}
\renewcommand{\le}{{\,\leqslant}}
\renewcommand{\ge}{{\,\geqslant}}

\newcommand{\1}{{\mathbf 1}}

\newcommand{\eps}{{\varepsilon}}

\begin{document}

\title[The linear sieve \& twin primes]{A modification of the linear sieve, and the\\ count of twin primes}

\author{Jared Duker Lichtman}
\address{Mathematical Institute, University of Oxford, Oxford, OX2 6GG, UK}

\email{jared.d.lichtman@gmail.com}



\date{February 14, 2024.}



\begin{abstract}
We introduce a modification of the linear sieve whose weights satisfy strong factorization properties, and consequently equidistribute primes up to size $x$ in arithmetic progressions to moduli up to $x^{10/17}$. This surpasses the level of distribution $x^{4/7}$ with the linear sieve weights from well-known work of Bombieri, Friedlander, and Iwaniec, and which was recently extended to $x^{7/12}$ by Maynard. As an application, we obtain a new upper bound on the count of twin primes. Our method simplifies the 2004 argument of Wu, and gives the largest percentage improvement since the 1986 bound of Bombieri, Friedlander, and Iwaniec.
\end{abstract}

\maketitle


\section{Introduction}
Given a finite set $\mathcal A$ of positive integers, sieve methods offer a broad framework for estimating the number of elements in $\mathcal A$ all whose prime factors exceed $z$, denoted $S(\mathcal A, z)$, in terms of the approximate density $g_{\mathcal A}(d)=g(d)$ of multiples of $d$ in $\mathcal A$, denoted $\mathcal A_d$. Note one often expects $S(\mathcal A, z)\approx |\mathcal A|\prod_{p<z}\big(1-g(p)\big)$. Combinatorial sieves may be viewed as refinements of the basic inclusion-exclusion principle, and are described by a sequence of weights $\lambda(d)\in\{-1,0,1\}$ supported on integers up to some level $D\ge1$. We refer the reader to Opera de Cribro \cite{Opera} for a more thorough introduction to the subject.

In particular, the upper bound weights $\lambda^+(d)$ for the linear sieve satisfy
\begin{align}\label{eq:linearJR}
S(\mathcal A, z) \ \le \ |\mathcal A|\prod_{p<z}\big(1-g(p)\big)\,\Big(F\big(\tfrac{\log D}{\log z}\big)+o(1)\Big) \ + \ \sum_{\substack{d\le D\\p\mid d\Rightarrow p<z}} \lambda^+(d) \, \big(|\mathcal A_d| - |\mathcal A|g(d)\big)
\end{align}
as $D\to\infty$, provided $g=g_{\mathcal A}$ satisfies some mild conditions. Here the function $F:\R_{\ge1}\to\R_{\ge1}$ is defined by a delay-differential equation, as in \eqref{eq:delaydiff}. For sets $\mathcal A$ sufficiently equidistributed in arithmetic progressions the second sum over $d\le D$ in \eqref{eq:linearJR} contributes negligibly, in which case the main term is $S(\mathcal A, z) \lesssim |\mathcal A|\prod_{p<z}\big(1-g(p)\big)F(s)$, where $z = D^{1/s}$. In fact, $F(s)\to1$ as $s\to\infty$ so the main term confirms the na\"ive expectation in this case. Moreover, $F$ is optimal in the sense that the bound \eqref{eq:linearJR} is attained sharply for a particular set $\mathcal A$. 

We introduce this sieve theory setup more formally in \S\ref{sec:sievesetup} below, and define the sieve weights $\lambda^+$ explicitly in \S\ref{sec:sievesupport}. See \cite[\S 12]{Opera} for further details on the linear sieve ($\beta=2$), as well as \cite[\S 11]{Opera} for its generalization to the $\beta$-sieve.

The linear sieve is powerful when combined with equidistribution estimates which make the final sum in \eqref{eq:linearJR} small. For example, the Bombieri--Vinogradov Theorem shows that for every $\eps,A>0$, letting $Q=x^{\frac{1}{2}-\eps}$ we have
\begin{align}\label{eq:BV}
\sum_{q\le Q}\sup_{(a,q)=1}\Big|\pi(x;q,a)-\frac{\pi(x)}{\phi(q)}\Big| \ \ll_{\eps,A} \ \frac{x}{(\log x)^A}.
\end{align}
So by taking $D=Q$, \eqref{eq:linearJR} can give a good upper bound when the set $\mathcal A$ is related to the primes, such as when $\mathcal A=\{p+2:p\le x\}$, in which case \eqref{eq:linearJR} gives an upper bound for the count of twin primes.

The estimate \eqref{eq:BV} may be viewed as an assertion of the Generalized Riemann Hypothesis on average over moduli up to $Q=x^{\frac{1}{2}-\eps}$. It remains an important open problem to extend the range to $Q=x^{\frac{1}{2}+\delta}$ for some fixed $\delta>0$. Indeed, Elliott and Halberstam \cite{EH} conjectured such an extension up to $Q=x^{1-\eps}$ for any $\eps>0$.

In some contexts it suffices to relax the setup in \eqref{eq:BV} in order to raise the level of distribution. In particular, in the case of a fixed residue class $a\in\Z$, and the absolute values replaced by well-factorable weights $\lambda(q)$ (c.f. Def. \ref{def:wellfactor}), the celebrated result of Bombieri--Friedlander--Iwaniec \cite{BFI1} raised the level up to $Q=x^{\frac{4}{7}-\eps}$,
\begin{align}\label{eq:BFIlinear}
\sum_{\substack{q\le Q\\(q,a)=1}} \lambda(q) \,\Big(\pi(x;q,a)-\frac{\pi(x)}{\phi(q)}\Big) \ \ll_{a,A,\eps} \ \frac{x}{(\log x)^A}.
\end{align}

While the linear sieve weights are not themselves well-factorable, Iwaniec \cite{Ilinear} constructed a well-factorable variant $\widetilde{\lambda}^+$ of the weights $\lambda^+$ (and so \eqref{eq:BFIlinear} holds with $\lambda=\widetilde{\lambda}^+$), which are only slightly altered from $\lambda^+$ so that $\widetilde{\lambda}^+$ enjoys an analogous linear sieve bound as in \eqref{eq:linearJR}, notably with an identical form of the main term,
\begin{align}\label{eq:linearwell}
S(\mathcal A, z) \ \le \ |\mathcal A|\prod_{p<z}\big(1-g(p)\big)\,\Big(F\big(\tfrac{\log D}{\log z}\big)+o(1)\Big) \ + \ \sum_{\substack{d\le D\\p\mid d\Rightarrow p<z}} \widetilde{\lambda}^+(d) \, \big(|\mathcal A_d| - |\mathcal A|g(d)\big).
\end{align}

The bound \eqref{eq:BFIlinear} stood for several decades, but quite recently Maynard \cite{JM2} managed to extend the level in \eqref{eq:BFIlinear} further to $Q=x^{\frac{7}{12}-\eps}$ in the case of the weights $\lambda=\widetilde{\lambda}^+$. Given the currently available equidistribution estimates for primes, we note the level $x^{\frac{7}{12}}$ is a natural barrier for these weights.

In this article, we modify the technical construction of the linear sieve weights to avoid this barrier, and thereby produce new sieve weights that induce stronger equidistribution estimates for primes.

\begin{theorem}\label{thm:etalinsieve}
Let $D = x^{\frac{10}{17}-\eps}$. There exists a sequence $\widetilde{\lambda}^*(d)\in\{-1,0,1\}$ satisfying:

\textnormal{(1) Equidistribution for primes:}
for any fixed $a\in\Z$, $A,\eps>0$, we have
\begin{align*}
\sum_{\substack{d\le D\\(d,a)=1}}\widetilde{\lambda}^*(d) \,\Big(\pi(x;d,a)-\frac{\pi(x)}{\phi(d)}\Big) \ \ll_{a,A,\eps} \ \frac{x}{(\log x)^A}.
\end{align*}

\textnormal{(2) Sieve upper bound:} for $s\ge1$, $z=D^{1/s}$, we have
\begin{align*}
S(\mathcal A, z) \ \le \ |\mathcal A|\prod_{p<z}\big(1-g(p)\big)\big(F^*(s) + o(1)\big) \ + \ \sum_{\substack{d\le D\\p\mid d\Rightarrow p<z}} \widetilde{\lambda}^*(d) \, \big(|\mathcal A_d| - |\mathcal A|g(d)\big),
\end{align*}
where $F^*(s) \, \le \, 1.000081\,F(s)$ when $1\le s\le 3$, for the linear sieve function $F$ as in \eqref{eq:delaydiff}.
\end{theorem}

The key feature of Theorem \ref{thm:etalinsieve} is to obtain equidistribution up to level $x^{\frac{10}{17}}$ at the cost of only a tiny loss in the main term. See Theorem \ref{thm:Iwaniecprogrammable} and Proposition \ref{prop:lvlalpha1} for full technical statements and additional variations that may be of independent interest.

\subsection{Application to twin primes} 

We expect that Theorem \ref{thm:etalinsieve} should give numerous improvements to sieve bounds related to the primes. As proof of concept in this direction, we give a new upper bound for the count of twin primes up to $x$, denoted $\pi_2(x)$. Recall Hardy and Littlewood \cite{HL} conjectured the asymptotic formula
\begin{align}
\pi_2(x) \ \sim \ \frac{2x}{(\log x)^2}\prod_{p>2}\frac{1-2/p}{(1-1/p)^{2}} =: \Pi(x).
\end{align}

\begin{theorem}\label{thm:twinbound}
As $x$ tends to infinity, we have
\begin{align*}
\pi_2(x) \ \lesssim \ 3.29956\,\Pi(x).
\end{align*}
\end{theorem}

Theorem \ref{thm:twinbound} gives a $2.94\%$ refinement from the previous record bound of Wu \cite{WuII}. For reference, this gives the largest percentage improvement since the work of Bombieri, Friedlander, and Iwaniec \cite{BFI1}. See below for a chronology of the known upper bounds on $\pi_2(x)/\Pi(x)$. Also see Siebert \cite{Sei}, Riesel--Vaughan \cite[Lemma 5]{RVaughan} for numerically explicit forms of Selberg's bound \cite{Selb}.

\vspace{1.5mm}

\begin{center}
\begin{tabular}{c|l|l}
Year & Author(s) & $\pi_2(x)/\Pi(x)\,\lesssim$\\
\hline    
1919 & Brun \cite{Brun} & O(1)\\
1947 & Selberg \cite{Selb} & 8\\
1964 & Pan \cite{Pan}  & 6\\
1966 & Bombieri--Davenport \cite{BD}  & 4\\
1978 & Chen \cite{Chen} & 3.9171\\
1983 & Fouvry--Iwaniec \cite{FI}  & $3.7777\cdots=34/9$\\
1984 & Fouvry \cite{Fouvry} & $3.7647\cdots=64/17$\\
1986 & Bombieri--Friedlander--Iwaniec \cite{BFI1} & 3.5\\
1986 & Fouvry--Grupp \cite{FG} & 3.454\\
1990 & Wu \cite{WuI} & 3.418\\
2003 & Cai--Lu \cite{CL} & 3.406\\
2004 & Wu \cite{WuII} & 3.39951
\end{tabular}
\end{center}

\vspace{1.mm}

The main ingredients for these results come from applying sieve bounds to the set $\mathcal A = \{p+2 : p\le x\}$, and using equidistribution of primes in arithmetic progressions to handle remainder terms. Bombieri--Davenport obtained $\pi_2(x)/\Pi(x)\,\lesssim\, 4$ as a consequence of the Bombieri--Vinogradov theorem \eqref{eq:BV} and a standard sieve upper bound of level $x^{\frac{1}{2}-\eps}$. More generally, if one proves level of distribution $x^{\theta-\eps}$ then one immediately obtains $\pi_2(x)/\Pi(x)\,\lesssim\, 2/\theta$. Bombieri--Friedlander--Iwaniec proved $\pi_2(x)/\Pi(x)\,\lesssim\, 7/2$ by the well-factorable variant \eqref{eq:BFIlinear} level of distribution $x^{\frac{4}{7}-\eps}$, together with the linear sieve with well-factorable remainder \eqref{eq:linearwell}.

The other key ingredient to subsequent improvements is the {\it switching principle}, introduced in Chen's celebrated result that there are infinitely many primes $p$ such that $p+2$ has at most two prime factors \cite{Chenthm}. The basic insight is to use a weighted sieve inequality to split the problem into multiple cases, apply sieve bounds to $\mathcal A = \{p+2 : p\le x\}$ in certain cases, and then reinterpret the remaining cases as new sieving problems for switched sets $\mathcal B = \{m-2 \le x\}$ where the numbers $m$ are constructed from $\mathcal A$ (as prescribed depending on the case).

\subsection{Outline of main ideas in Theorem \ref{thm:etalinsieve}}

Maynard's new equidistribution results show equidistribution of the primes with sieve weights $\widetilde{\lambda}^+(d)$, provided $d=p_1\cdots p_r$ is restricted to suitably well-factorable integers. Unfortunately, the original linear sieve weights only partially satisfy these well-factorable conditions. In particular for $\eta>0$, when looking at the linear sieve of level $x^{\frac{7}{12}+\eta}$, some integers $d$ in its support do not satisfy the conditions, which means that $x^{\frac{7}{12}}$ is the limit for the linear sieve given our current equidistribution technology. Nevertheless, the key observation here is that only a few exceptional $d$ fail to satisfy these conditions. Moreover up to level $x^{\frac{10}{17}}$, i.e. $\eta<1/204$, the anatomy of exceptional $d$ may be precisely characterized in terms of $\eta$ (given specifically as $\mathcal P_4$, $\mathcal P_6$ in \eqref{eq:defP4P6}). In particular, as $\eta>0$ grows the family of exceptional integers contribute $O(\eta^5)$ to the sieve bound. However, we note this characterization breaks down when $\eta\ge 1/204$, and the contribution becomes considerably larger and more complicated.
 
As such we carefully revise the construction of the linear sieve, altering a few particular inclusion-exclusion steps in order to avoid the exceptional integers $d$ with bad factorizations. Once these terms no longer contribute to the sieve, this produces a worse and more complicated main term, but since there are only a very small number of such terms the resulting loss is small. And since these modified weights now satisfy stronger factorization properties in their support, we can now leverage the full strength of Maynard's equidistribution results.

\section*{Notation}
We use the Vinogradov $\ll$ and $\gg$ asymptotic notation, and the big oh $O(\cdot)$ and $o(\cdot)$ asymptotic notation. We use $f\sim g$, $f \lesssim g$, and $f \gtrsim g$ to denote $f = (1+o(1)) g$, $f \le (1+o(1)) g$, and $f \ge (1+o(1)) g$, respectively. Dependence on a parameter will be denoted by a subscript.

The letter $p$ will always be reserved to denote a prime number, $\pi(x)$ is the prime counting function, and $\pi(x;d,a)$ is the count of primes up to $x$ congruent to $a$ (mod $d$). We use $\varphi$ to denote the Euler totient function, $\mu$ the M\"obius function, and $e(x) := e^{2\pi ix}$ the complex exponential. We use $\1$ to denote the indicator function of a statement. For example, for a set $A$ denote
\[\1_{a\in A} = 
\begin{cases}
1, & \text{if }a\in A,\\
0, & \text{else.}
\end{cases},
\qquad
\1_{a_1,\ldots,a_i\in A}^{a_0\notin A} = 
\begin{cases}
1, & \text{if }a_1,\ldots,a_i\in A \ \text{and }a_0\notin A,\\
0, & \text{else.}
\end{cases}\]

Finally, we refer to various sieve weights referred places throughout the article, so we take a moment to list them here:

We generically write $\lambda$ to denote a sequence of weights in $\{0,\pm1\}$. In particular, $\lambda^+$ and $\lambda^-$ refer to the (upper and lower bound) weights of the linear sieve, given by restrictions of the M\"obius function, $\lambda^\pm (d) =\mu(d)\1_{d\in\mathcal D^\pm}$. Analogously, the modified (upper bound) linear sieve weights $\lambda^*$ are given by $\lambda^* (d) =\mu(d)\1_{d\in\mathcal D^*}$. Here the support sets $\mathcal D^\pm$ and $\mathcal D^*$ are defined in \eqref{eq:Dplusdef} and \eqref{eq:Dstardef}.
We also write $\lambda^{(r)}$ to refer to the weights $\lambda^+$ or $\lambda^-$ (and $\mathcal D^{(r)}$ to refer to $\mathcal D^+$ or $\mathcal D^-$), depending on whether $r$ is odd or even.

The well-factorable weights $\widetilde{\lambda}^\pm$ are defined in \eqref{eq:lambdapm}. Following Iwaniec, this construction involves certain auxiliary weights at intermediate steps, namely, $\lambda_{(D_1,\ldots, D_r)}$ defined in \eqref{eq:lamD1D},
and $\lambda_{(D_1,\ldots, D_r)}^{(r)} = \lambda_{(D_1,\ldots, D_r)} \ast \lambda^{(r)}$ defined in \eqref{eq:lamD1Dr}. The analogous construction starting from $\lambda^*$ gives modified weights $\widetilde{\lambda}^*$, defined in \eqref{def:tildelambdastar}.

\section{Technical setup and results}


\subsection{Factorization of weights and their level of distribution}

\begin{definition}[well-factorable]\label{def:wellfactor}
Let $Q\in\R_{\ge1}$. A sequence $\lambda(q)$ is {\bf well-factorable of level $Q$}, if for every factorization $Q=Q_1Q_2$ into $Q_1,Q_2\in\R_{\ge1}$, there exist sequences $\gamma_1,\gamma_2$ such that
\begin{enumerate}
\item $|\gamma_1(q_1)|,|\gamma_2(q_2)| \le 1$ for all $q_1,q_2\in\N$,
\item $\gamma_i(q) =0$ if $q\notin [1,Q_i]$ for $i=1,2$,
\item We have $\lambda=\gamma_1\ast\gamma_2$, i.e.,
\begin{align*}
\lambda(q) = \sum_{q=q_1q_2} \gamma_1(q_1)\gamma_2(q_2).
\end{align*}
\end{enumerate}
\end{definition}

In \cite[Theorem 10]{BFI1}, Bombieri--Friedlander--Iwaniec established level of distribution $x^{\frac{4}{7}-\eps}$ with well-factorable weights.
\begin{theorem}[Bombieri--Friedlander--Iwaniec \cite{BFI1}] \label{thm:BFIwell}
Fix any $a\in\Z$ and let $A,\eps>0$. For any well-factorable sequence $\lambda$ of level $Q\le x^{\frac{4}{7}-\eps}$, we have
\begin{align*}
\sum_{\substack{q\le Q\\(q,a)=1}}\lambda(q) \,\Big(\pi(x;q,a)-\frac{\pi(x)}{\phi(q)}\Big) \ \ll_{a,A,\eps} \ \frac{x}{(\log x)^A}.
\end{align*}
\end{theorem}

Maynard \cite{JM2} considered a natural strengthening of well-factorable sequences.

\begin{definition}[triply well-factorable]
Let $Q\in\R_{\ge1}$. A sequence $\lambda(q)$ is {\bf triply well-factorable of level $Q$}, if for every factorization $Q=Q_1Q_2Q_3$ into $Q_1,Q_2,Q_3\in\R_{\ge1}$, there exist sequences $\gamma_1,\gamma_2,\gamma_3$ such that
\begin{enumerate}
\item $|\gamma_1(q_1)|,|\gamma_2(q_2)|,|\gamma_3(q_3)| \le 1$ for all $q_1,q_2,q_3\in\N$,
\item $\gamma_i(q) =0$ if $q\notin [1,Q_i]$ for $i=1,2,3$,
\item We have $\lambda=\gamma_1\ast\gamma_2\ast\gamma_3$, i.e.,
\begin{align*}
\lambda(q) = \sum_{q=q_1q_2q_3} \gamma_1(q_1)\gamma_2(q_2)\gamma_3(q_3).
\end{align*}
\end{enumerate}
\end{definition}

The definitions of well-factorable and triply well-factorable sequences are quite natural and relatively simple from a conceptual standpoint. In \cite[Theorem 1.1]{JM2}, Maynard obtains powerful equidistribution results for triply well-factorability that are beyond the scope of well-factorability. Unfortunately, triply well-factorability is too restrictive a condition for us to produce Theorem \ref{thm:etalinsieve}. As such we are forced to identify the precise mechanism that enables Maynard's equidistribution results, and extract the following technical definition that is implicit in \cite{JM2}.\footnote{
Indeed, the definition of programmably factorable in the special case $Q_3=1$ gives the implicit condition (which is implied by well-factorable) that enables Bombieri--Friedlander--Iwaniec to get equidistribution \eqref{eq:BFIlinear}. Also see Lemma 5 in \cite{FG}.}

\begin{definition}[programmably factorable]
Let $0<\delta<10^{-5}$. For $x\in\R_{>1}$, a sequence $\lambda(q)$ is {\bf programmably factorable of level $Q$ (relative to $x$, $\delta$)}, if for every $N\in [x^{2\delta},x^{\frac{1}{3}+\delta/2}]$ there exists a factorization $Q=Q_1Q_2Q_3$ with $Q_1,Q_2,Q_3\in\R_{\ge1}$, satisfying the system
\begin{align}\label{eq:defprogram}
Q_1 &\le\; Nx^{-\delta}, \nonumber\\
N^2 Q_2 Q_3^2 &\le\; x^{1-\delta},\\
N^2 Q_1 Q_2^4 Q_3^3 &\le\; x^{2-\delta}, \nonumber\\
N Q_1 Q_2^5 Q_3^2 &\le\; x^{2-\delta}. \nonumber
\end{align}
And for every such factorization $Q=Q_1Q_2Q_3$ there exist sequences $\gamma_1,\gamma_2,\gamma_3$ such that
\begin{enumerate}
\item $|\gamma_1(q_1)|,|\gamma_2(q_2)|,|\gamma_3(q_3)| \le 1$ for all $q_1,q_2,q_3\in\N$,
\item $\gamma_i(q) =0$ if $q\notin [1,Q_i]$ for $i=1,2,3$,
\item We have $\lambda=\gamma_1\ast\gamma_2\ast\gamma_3$, i.e.,
\begin{align*}
\lambda(q) = \sum_{q=q_1q_2q_3} \gamma_1(q_1)\gamma_2(q_2)\gamma_3(q_3).
\end{align*}
\end{enumerate}
\end{definition}

Programmable factorability is the key technical definition in this article. It is named in allusion to the linear programming-type system of inequalities \eqref{eq:defprogram} that the factors satisfy. The diagram below displays the various implications among the definitions.

\vspace{3mm}
\begin{align*}
\lambda \ & \text{is {\bf triply well-factorable} of level }Q.  \qquad \Longrightarrow \qquad  \lambda \text{ is {\bf well-factorable} of level }Q.\\
& \\
& \qquad\qquad\qquad\Big\Downarrow\\
& \\
\lambda \ & \text{is {\bf programmably factorable} of level }Q \text{ (relative to $x$, $\delta$)}.
\end{align*}
\vspace{3mm}

In the key result \cite[Theorem 1.1]{JM2}, Maynard extended the level of distribution up to $Q<x^{\frac{3}{5}}$ for programmably factorable weights. Note that level $x^{\frac{3}{5}}$ is the natural barrier for \eqref{eq:defprogram} to admit a solution.

\begin{theorem}[Maynard \cite{JM2}]\label{thm:Maynardprogram}
Fix any $a\in\Z$ and let $A,\eps>0$. For any programmably factorable sequence $\lambda$ of level $Q\le x^{\frac{3}{5}-\eps}$ (relative to $x$, $\eps/50$), we have
\begin{align*}
\sum_{\substack{q\le Q\\(q,a)=1}}\lambda(q) \,\Big(\pi(x;q,a)-\frac{\pi(x)}{\phi(q)}\Big) \ \ll_{a,A,\eps} \ \frac{x}{(\log x)^A}.
\end{align*}
\end{theorem}
\begin{remark}
\cite[Theorem 1.1]{JM2} was stated for triply-factorable sequences, but its proof in fact gives the result for programmably factorable sequences.
\end{remark}

Note the weights $\widetilde{\lambda}^+$ are composed of well-factorable---but not neccessarily programmably factorable---sequences of given level $D$. Nevertheless, Maynard showed the upper bound weights $\widetilde{\lambda}^+$ of sieve level $D=x^{\frac{7}{12}-\eps}$ are programmably factorable of level $D\le Q=x^{\frac{3}{5}-\eps}$ (relative to $x$, $\eps/50$). By Theorem \ref{thm:Maynardprogram} this gives \cite[Theorem 1.2]{JM2} below.

\begin{corollary}[Maynard \cite{JM2}]\label{cor:May712}
For any fixed $a\in\Z$ and $A,\eps>0$, the weights $\widetilde{\lambda}^+$ from
\eqref{eq:wellwts} of sieve level $D= x^{\frac{7}{12}-\eps}$ satisfy
\begin{align*}
\sum_{\substack{d\le D\\(d,a)=1}} \widetilde{\lambda}^+(d) \,\Big(\pi(x;d,a)-\frac{\pi(x)}{\phi(d)}\Big) \ \ll_{a,A,\eps} \ \frac{x}{(\log x)^A}.
\end{align*}
\end{corollary}

Later, in Proposition \ref{prop:lvlalpha1}, we shall obtain technical improvements of Corollary \ref{cor:May712} for Iwaniec's weights $\widetilde{\lambda}^\pm$ (both upper and lower), in special cases where equidistribution is restricted to moduli which are smooth, or otherwise amenable to programmable factorization.

We may summarize the definitions and results of the section up to this point as follows:
\vspace{3mm}
\begin{align*}
\lambda \ &\text{is {\bf triply well-factorable} of level }Q.  \qquad \qquad \quad  \lambda \ \text{is {\bf well-factorable} of level }Q.\\
& \ \bullet \text{equidistributed for }Q< x^{\frac{3}{5}}.   
\qquad\qquad\qquad\qquad\quad \bullet \text{equidistributed for }Q< x^{\frac{4}{7}}.\\
& \ \bullet \text{can take }\lambda = \widetilde{\lambda}^+ \text{ for }D\le\, Q^{\frac{2}{3}}.   
\qquad\qquad\qquad\qquad \bullet \text{can take }\lambda = \widetilde{\lambda}^+ \text{ for }D\le\, Q.\\
& \\
\lambda \ & \text{is {\bf programmably factorable} of level }Q \text{ (relative to $x$, $\delta$)}.\\
& \ \bullet\text{equidistributed for }Q< x^{\frac{3}{5}}.\\
& \ \bullet \text{can take }\lambda = \widetilde{\lambda}^+ \text{ for }D\le\, Q< x^{\frac{7}{12}}.   
\end{align*}
\vspace{3mm}

For each type of sequence, we have outlined their corresponding levels of distribution, and the levels at which the type is satisfied by the upper bound weights for the linear sieve. Observe well-factorability is flexible enough to accommodate the linear sieve to any level, but has weaker equidistribution. On the other hand, triple well-factorability has stronger equidistribution, but is too rigid to accommodate the linear sieve (at nontrivial levels). Finally, programmable factorability also has strong equidistribution in addition to (nontrivially) accommodating the linear sieve, though at the cost of conceptual technicality.

\begin{remark} 
In general, $\lambda$ well-factorable of level $Q$ directly implies $\lambda$ triply well-factorable of level $Q^{\frac{2}{3}}$.\footnote{Indeed, take any factorization $Q=Q_1Q_2Q_3$, with (say) $Q_1\ge Q_2\ge Q_3 \ge 1$. Note $Q_3\le Q^{1/3}$.
If $\lambda$ is well-factorability of level $Q/Q_3=Q_1Q_2$, there are sequences $\gamma_1, \gamma_2$ supported on $[1,Q_1]$, $[1,Q_2]$ with $\lambda = \gamma_1\ast \gamma_2 = \gamma_1\ast \gamma_2\ast \delta$. Here $\delta(q)=\1_{q=1}$. Hence $\lambda$ is triply well-factorable of level $\inf_{Q=Q_1Q_2Q_3}Q/Q_3 \ge Q^{2/3}$.}
In particular, for $\lambda=\widetilde{\lambda}^+$ the triply well-factoble level $Q^{\frac{2}{3}}< x^{\frac{2}{5}}$ is sharp.\footnote{Indeed, consider the factorization $Q=Q_1Q_2Q_3$ with $(Q_1,Q_2,Q_3)=(Q^{\frac{1}{3}-\eps},Q^{\frac{1}{3}-\eps},Q^{\frac{1}{3}+2\eps})$. Then for $q=p_1p_2p_3$ of size $p_1,p_2\sim Q^{\frac{1}{3}}$, $p_3\sim Q^{\frac{1}{9}}$, we see $p_1,p_2 > Q_1=Q_2$. Thus all sequences $\gamma_i$ supported on $Q_i$ satisy $\gamma_1\ast \gamma_2\ast \gamma_3(q) = 0$. In particular $\gamma_1\ast \gamma_2\ast \gamma_3\neq \widetilde{\lambda}^+$. }
\end{remark}

\subsection{Sieve theory setup and bounds} \label{sec:sievesetup}

\par
\vspace{.5mm}

We recall the standard sieve-theoretic notation. Given a finite set $\mathcal{A}\subset \N$, set of primes $\mathcal P$, and a threshold $z>0$, we define $\mathcal{A}_d = \{n\in \mathcal{A}: d\mid n\}$ and remainder $r_\mathcal{A}$ via
\begin{align*}
|\mathcal{A}_d| = g(d)|\mathcal A| + r_\mathcal{A}(d),
\end{align*}
where $g$ is a multiplicative function, with $0\le g(p)<1$ for $p\in \mathcal P$ (we assume $g(p)=0$ if $p\notin\mathcal P$). Also define $P(z) = \prod_{p<z,p\in \mathcal P}p$ and $V(z) = \prod_{p\mid P(z)}\big(1-g(p)\big)$. The central object of interest is the {\it sifted sum} 
\begin{align}
S(\mathcal{A},z) = S(\mathcal{A},\mathcal P,z) = \sum_{n\in \mathcal{A}}\1_{(n,P(z))=1}.
\end{align}

Later for our application of interest, we will set $g(d)=1/\phi(d)$. For now, it suffices for us to assume for all $2\le w\le z$,
\begin{align}\label{eq:Vwz}
\frac{V(w)}{V(z)} = \prod_{\substack{w\le p<z\\ p\in\mathcal P}}\big(1-g(p)\big) \ = \ \frac{\log z}{\log w}\bigg(1 + O\Big(\frac{1}{\log w}\Big)\bigg).
\end{align}
\begin{remark}
The proof of the upper bound for the standard linear sieve only requires a one-sided inequality for $V(w)/V(z)$, whereas our modification requires the above two-sided condition \eqref{eq:Vwz}.
\end{remark}

The basic result which we shall adapt is the linear sieve with well-factorable remainder, as in \cite[Theorem 12.20]{Opera}.
\begin{theorem}[Iwaniec \cite{Opera}]\label{thm:wellfactor}
Let $\eps>0$ and $D>1$ be sufficiently small and large, respectively. Then for $s \ge 1$ and $z=D^{1/s}$, we have
\begin{align*}
S(\mathcal A, z) \ & \le \ |\mathcal A|V(z)\big(F(s) + O(\eps)\big) \ + \ \sum_{d\mid P(z)}\widetilde{\lambda}^+(d)\, r_{\mathcal A}(d),\\
S(\mathcal A, z) \ & \ge \ |\mathcal A|V(z)\big(f(s) + O(\eps)\big) \ - \ \sum_{d\mid P(z)}\widetilde{\lambda}^-(d)\, r_{\mathcal A}(d),
\end{align*}
where the implied constant only depends that of \eqref{eq:Vwz}. Here the weights $\widetilde{\lambda}^\pm$ are
\begin{align}\label{eq:wellwts}
\widetilde{\lambda}^\pm(d) = \sum_{j\le\exp(\eps^{-3})}\lambda_j^\pm(d)
\end{align}
for some well-factorable sequences $\lambda_j^\pm$ of level $D$. The functions $F,f:\R^+\to\R$ satisfy the system of delay-differential equations
\begin{align}\label{eq:delaydiff}
sF(s) & = 2e^\gamma \qquad[s\le3] & \big(sF(s)\big)' & = f(s-1), \nonumber\\
sf(s) & = 0 \quad \qquad[s\le2]  & \big(sf(s)\big)' & = F(s-1).
\end{align}
\end{theorem}
\begin{remark}
See Iwaniec \cite[Theorem 1]{Ilinear} for an alternate formulation and proof, which gives sharper quantitative bounds than $O(\eps)$. However, it is more technical than necessary for our purposes.
\end{remark}

The main result of this article is the following modification of the linear sieve with programmably factorable remainder.

\begin{theorem}\label{thm:Iwaniecprogrammable}
Let $\mathcal A$ be a finite set of positive integers with density function $g(d)$ satisfying \eqref{eq:Vwz}, and $F(s)$ the function defined by the system \eqref{eq:delaydiff}. Let $\eps>0$ and $x>1$ be sufficiently small and large, respectively. Then for $\eta\ge0$, $D=x^{\frac{7}{12}+\eta}$, $s \ge 1$, and $z=D^{1/s}$, we have
\begin{align*}
S(\mathcal A, z) \ \le \ |\mathcal A|V(z)\big(F^*(s) + O(\eps)\big) \ + \ \sum_{d\mid P(z)}\widetilde{\lambda}^*(d)\, r_{\mathcal A}(d),
\end{align*}
where the implied constant only depends that of \eqref{eq:Vwz}. Here the weights $\widetilde{\lambda}^*$ are
\begin{align}\label{eq:programwts}
\widetilde{\lambda}^*(d) = \sum_{j\le\exp(\eps^{-3})}\lambda_j^*(d)
\end{align}
for some programmably factorable sequences $\lambda_j^*$ of level $D$ (relative to $x$, $\eps/50$). For $\eta<\frac{1}{204}$ we have $F^*(s) = F(s) + O(\eta^5)$, and $F^*(s) \, \le \, 1.000081\,F(s)$ for $1\le s\le 3$, $\eta=\frac{1}{204}$.
\end{theorem}

Note Theorem \ref{thm:Maynardprogram}, applied to each $\lambda=\lambda_j^*$ above, immediately implies the following.

\begin{corollary}\label{cor:program}
Given any fixed $a\in\Z$ and $A$. For $\eta<\frac{1}{204}$, the weights $\widetilde{\lambda}^*$ as in \eqref{eq:programwts} of level $D=x^{\frac{7}{12}+\eta}$ satisfy
\begin{align*}
\sum_{\substack{d\le D\\(d,a)=1}}\widetilde{\lambda}^*(d) \,\Big(\pi(x;d,a)-\frac{\pi(x)}{\phi(d)}\Big) \ \ll_{a,A,\eta} \ \frac{x}{(\log x)^A}.
\end{align*}
\end{corollary}

\section{Programmably factorable support}\label{sec:sievesupport}

The upper and lower bound weights $\lambda^\pm$ for the linear sieve of level $D$ are defined by
\begin{align*}
\lambda^\pm(d) = \mu(d)\1_{d\in \mathcal D^\pm},
\end{align*}
where $\mathcal D^\pm=\mathcal D^\pm(D)$ are the standard support sets
\begin{align}\label{eq:Dplusdef}
\mathcal D^+ & =  \{ p_1\cdots p_r \; : \; D^{1/2} \ge p_1 \ge \cdots \ge p_r, \text{ and } p_1\cdots p_{l-1}p_l^3 \le D \ \text{ for each odd }l \le r\},\\
\mathcal D^- & =  \{ p_1\cdots p_r \; : \; D^{1/2} \ge p_1 \ge \cdots \ge p_r,\text{ and }
p_1\cdots p_{l-1}p_l^3 \le D \ \text{ for each even }l \le r\}.\nonumber
\end{align}
We may also write $\mathcal D^{(r)}$ to denote $\mathcal D^+$ or $\mathcal D^-$, when $r$ is even or odd, respectively.

Observe that both sets satisfy the containment $\mathcal D^\pm(D) \subset \mathcal D^{\textnormal{well}}(D)$, where
\begin{align}\label{eq:defDwell}
\mathcal D^{\textnormal{well}} =  \{ p_1\cdots p_r \; : \; D^{1/2} \ge p_1 \ge \cdots \ge p_r,\text{ and }
p_1\cdots p_{l-1}p_l^2 \le D \ \text{ for each }l \le r\}.
\end{align}
We shall return to this observation later in the section.

In \cite{JM2}, Maynard deduces Corollary \ref{cor:May712} for $\widetilde{\lambda}^+$ from the general Theorem \ref{thm:Maynardprogram} by means of the following key result \cite[Proposition 9.1]{JM2} (along with a construction of Iwaniec we shall address in later sections), which programmably factorizes elements of the support $\mathcal D^+$.

\begin{proposition}[Maynard \cite{JM2}]
Let $0<\delta<10^{-3}$ and let $D=x^{\frac{7}{12}-50\delta}$, $N\in [x^{2\delta},x^{\frac{1}{3}+\delta/2}]$. Then every $d\in \mathcal D^+(D)$ has a factorization $d=d_1d_2d_3$ such that $d_1 \le\; Nx^{-\delta}$ and
\begin{equation}
\begin{split}\label{eq:programmsupport+}
N^2 d_2d_3^2 &\le\; x^{1-\delta},\\
N^2 d_1d_2^4d_3^3 &\le\; x^{2-\delta},\\
N d_1d_2^5d_3^2 &\le\; x^{2-\delta}.
\end{split}
\end{equation}
\end{proposition}
\begin{remark}
The level $x^{\frac{7}{12}}$ is sharp in this construction. Indeed, {\it heuristically speaking}, the linear sieve weights are not programmably factorable of level $D=x^{\frac{7}{12}+\eta}$ for any $\eta>0$, because the support set contains obstructing (families of) elements $d\in\mathcal D^+(D)$ of the form $d=p_1\cdots p_r$ where $p_1\approx \cdots \approx p_6 \approx D^{\frac{1}{7}}$, or where $p_1\approx p_2 \approx D^{\frac{2}{7}}$ and $p_3\approx p_4\approx D^{\frac{1}{7}}$. This heuristic description of the obstructions is made precise by the families $\mathcal P_4$, $\mathcal P_6$ (in \eqref{eq:defP4P6} below), and thereby tells us how we should restrict the support set in order to increase the level (namely, to $\mathcal D^*$ in \eqref{eq:Dstardef} below).
\end{remark}

For $\eta>0$, level $D = x^{\frac{7}{12}+\eta}$, we define the modified weights $\lambda^*=\lambda_\eta^*$,
\begin{align}\label{eq:lambstardef}
\lambda^*(d) = \mu(d)\1_{d\in \mathcal D^*},
\end{align}
for the support set $\mathcal D^*$,
\begin{align}\label{eq:Dstardef}
\mathcal D^* \ = \  \mathcal D^+(x^{\frac{7}{12}}) \cup \{p_1\cdots p_r\in\mathcal D^+(x^{\frac{7}{12}+\eta}) : p_1\cdots p_i\notin \mathcal P_i, i\le r, i\in\{4,6\}\}.
\end{align}
Here $\mathcal P_4$ and $\mathcal P_6=\mathcal P_{6,1}\cup \mathcal P_{6,2}$ are exceptional subsets of $\mathcal D^+(x^{\frac{7}{12}+\eta})$, given by
\begin{align}\label{eq:defP4P6}
\mathcal P_4 &= \{p_1\cdots p_4 \; : \;  p_1 < x^{\frac{1}{6}+2\eta} \ \text{ and } \ p_2p_4>x^{\frac{1}{4}-3\eta}\}, \nonumber\\
\mathcal P_{6,1} &= \{p_1\cdots p_6 \; : \; p_1p_2 < x^{\frac{1}{6}+2\eta} \ \text{ and } \  p_2p_3p_4>x^{\frac{1}{4}-3\eta} \ \text{ and } \ p_6>x^{\frac{1}{12}-5\eta}\},\\
\mathcal P_{6,2} &= \{p_1\cdots p_6 \; : \; p_1,\;p_2p_3 < x^{\frac{1}{6}+2\eta}\ \text{ and } \ p_1p_4,\;p_2p_3p_4 > x^{\frac{1}{4}-3\eta}  \ \text{ and } \ p_6>x^{\frac{1}{12}-5\eta}\}. \nonumber
\end{align}

The modified support set $\mathcal D^* = \mathcal D^*_\eta$ is understood to depend on $\eta>0$ (as do $\mathcal P_4$, $\mathcal P_6$), but we will suppress this for notational convenience.

In this section, we establish a programmable factorization of the elements of the support $\mathcal D^*$ provided $D<x^{\frac{10}{17}}$, i.e. $\eta < \frac{1}{204}$. This will serve as the key technical input for the proof of Theorem \ref{thm:Iwaniecprogrammable}.

\begin{proposition}[Factorization of elements of $\mathcal D^*$]\label{prop:factorD}
Let $0<\delta < 10^{-5}$, and take $0<\eta<\frac{1}{204}-3\delta$ and $N\in [x^{\delta}, x^{\frac{1}{3}-\delta/2}]$. If $d\in \mathcal D^*$ for $D = x^{\frac{7}{12}+\eta-50\delta}$, then we may factor $d=d_1 d_2 d_3$ such that $d_1\le Nx^{-\delta}$ and
\begin{equation}\label{eq:JMfactorable}
\begin{split}
N^2 \ d_2 d_3^2 & \ \le \ x^{1-\delta},\\
N^2 d_1 d_2^4 d_3^3 & \ \le \ x^{2-\delta},\\
N d_1 d_2^5 d_3^2 & \ \le \ x^{2-\delta}.
\end{split}
\end{equation}
\end{proposition}
On the first attempt working through technicalities, we encourage the reader to set $\delta=0$ in order to better view the key features.

Before proving the proposition, we need some lemmas. The first gives a general-purpose criterion to factor an integer $d$.

\begin{lemma}\label{lem:criterion}
Let $D = x^{\frac{7}{12}+\eta-50\delta}$ for $-\frac{1}{84}<\eta<\frac{1}{60}$. A factorization $d=d_1d_2d_3$ satisfies \eqref{eq:JMfactorable}, provided $d_1,d_2,d_3\ge1$ satisfy
\begin{gather}\label{eq:criterion}
d_2\in[x^{\frac{1}{6}+2\eta},x^{\frac{1}{4}-3\eta}],\\
d_1\le Nx^{-\delta}, \quad \text{and} \quad d_3\le D/Nd_2. \nonumber
\end{gather}
\end{lemma}
\begin{proof}
By \eqref{eq:criterion}, $Nd_3\le D/d_2$ and so
\begin{align*}
N^2 \ d_2 d_3^2 & \ \le \ D^2/d_2 \le x^{2(\frac{7}{12}+\eta-50\delta)-(\frac{1}{6}+2\eta)} = x^{1-\delta},\\
N^2 d_1 d_2^4 d_3^3 & \ \le \ D^3 d_2 \le x^{3(\frac{7}{12}+\eta-50\delta)+(\frac{1}{4}-3\eta)}
= x^{2-\delta},\\
N d_1 d_2^5 d_3^2 & \ \le \  D^2 d_2^3 \le 
x^{2(\frac{7}{12}+\eta-50\delta) + 3(\frac{1}{4}-3\eta)} = x^{\frac{23}{12}-7\eta-100\delta} < x^{2-\delta},
\end{align*}
using $\eta\in (-\frac{1}{84},\frac{1}{60})$. This gives \eqref{eq:JMfactorable}.
\end{proof}

The above criterion implies factorizations in the following special cases.

\begin{lemma}\label{lem:reduce}
Let $D = x^{\frac{7}{12}+\eta-50\delta}$ for $\eta<\frac{1}{60}$. For $r\ge4$, let $x^{\frac{1}{6}+2\eta} > p_1\ge\cdots\ge p_r$ be primes for which $d=p_1\cdots p_r\in \mathcal D^+(D)$. Suppose $d_2$ is one of the subproducts in $\{p_1p_4,p_2p_3,p_2p_4,p_2p_3p_4\}$. Then $d$ has a factorization $d=d_1d_2d_3$ satisfying \eqref{eq:JMfactorable}, provided
\begin{align*}
d_2\in [x^{\frac{1}{6}+2\eta},x^{\frac{1}{4}-3\eta}].
\end{align*}
\end{lemma}
\begin{proof}
Let $C=D/Nd_2$ and note either $p_1\le N$ or $p_1\le C$, since
\begin{align*}
p_1^2\le D^{\frac{2}{3}} = x^{\frac{2}{3}(\frac{7}{12}+\eta-50\delta)} < x^{\frac{1}{3}+4\eta-50\delta}\le D/d_2=NC.
\end{align*}

We proceed by induction on $r\ge4$. As the base case $r=4$, by Lemma \ref{lem:criterion} it suffices for each $b$ to factor $p_1\cdots p_4/d_2=d_1d_3$ for $d_1\le N, d_3\le C$. Indeed, this holds when $d_2=p_2p_3p_4$ since $p_1^2\le AC$, and similarly:
\begin{itemize}
\item If $d_2=p_2p_4$ then $p_3^2 \le D/p_1p_2p_3\le NC/p_1$ implies $p_1p_3=d_1d_3$ for some $d_1\le N$, $d_3 \le C$.
\item If $d_2=p_1p_4$ then $p_3^2 \le D/p_1p_2p_3\le NC/p_2$ implies $p_2p_3=d_1d_3$ for some $d_1\le N$, $d_3 \le C$.
\item If $d_2=p_2p_3$ then $p_4^2 \le D/p_1p_2p_3\le NC/p_1$ implies $p_1p_4=d_1d_3$ for some $d_1\le N$, $d_3 \le C$.
\end{itemize}

Now for $r\ge5$, we inductively assume a factorization $p_1\cdots p_{r-1} = d_1d_2d_3$ with $d_1\le N$, $d_3\le C$. Then $p_r^2 \le D/p_1\cdots p_{r-1} = NC/(ac)$ so either $d_1p_r\le N$ or $d_3 p_r\le C$, extending the factorization. Hence Lemma \ref{lem:criterion} applies again, and completes the proof.
\end{proof}

Finally, if the primes dividing $d$ are small enough, we may use the greedy algorithm to factor $d$ as follows.

\begin{lemma}\label{lem:greedy}
Let $D = x^{\frac{7}{12}+\eta-50\delta}$ for $\eta<\frac{1}{60}$. For $r\ge4$, let $x^{\frac{1}{6}+2\eta} > p_1\ge\cdots\ge p_r$ be primes for which $d=p_1\cdots p_r\in \mathcal D^+(D)$, and $p_6 < x^{\frac{1}{12}-5\eta}$ if $r\ge6$. Then $d$ has a factorization $d=abc$ satisfying \eqref{eq:JMfactorable}, provided there is a factorization $p_1p_2p_3p_4 = d_1d_2d_3$ satisfying
\begin{align}\label{eq:greedy}
d_1\le Nx^{-\delta}, \quad d_3\le x^{1-2\delta}/DN,\quad d_2\le D^2/x^{1-3\delta} = x^{\frac{1}{6}+2\eta+3\delta}.
\end{align}
\end{lemma}
\begin{proof}
Let $D_1=Nx^{-\delta}$, $D_2 = D^2/x^{1-3\delta}$, $D_3 = x^{1-2\delta}/(DN)$, so that $d_i\le D_i$ by assumption.

We now greedily append primes to $d_i$ while preserving $d_i\le D_i$ for all $i$, i.e. where at the $j$th step we replace $d_i\mapsto d_ip_j$ (for one of $i=1,2,3$) provided $d_ip_j\le D_i$. Starting from $j=5$, we stop either when we have exhausted all primes (i.e. $j=r$), or $d_ip_j> D_i$ for each $i=1,2,3$. In the former case, we have the desired $d_1d_2d_3 = d= p_1\cdots p_r$ and $d_i\le D_i$ so we easily get
\begin{align*}
D_1 &= Nx^{-\delta},\\
N^2 D_2D_3^2 & = x^{1-\delta},\\
N^2 D_1D_2^4D_3^3 & =  D^5 x^{-1+5\delta} \le x^{5\cdot\frac{3}{5}-1-245\delta} < x^{2-\delta},\\
N D_1D_2^5D_3^2 & = D^8 x^{-3+10\delta} \le x^{8\cdot\frac{3}{5}-3-390\delta} < x^{2-\delta}.
\end{align*}
Thus $d_1d_2d_3 = d= p_1\cdots p_r$ gives the desired factorisation.

In the latter case, there exists a terminal index $j<r$ for which $d_ip_j> D_i$ for all $i=1,2,3$. Note if $j$ is odd, then $d_i p_j \le  D_i$ for some $i$, since
\begin{align*}
p_j^3 \le \frac{D}{p_1\cdots p_{j-1}} = \frac{D_1D_2D_3}{d_1d_2d_3}.
\end{align*}
So the terminal $j$ is even with $j\ge6$. By assumption $p_j\le p_6\le x^{\frac{1}{12}-5\eta}$ is smaller than the width of the interval $[x^{\frac{1}{6}+2\eta},x^{\frac{1}{4}-3\eta}]$. And since $d_2\le D_2=x^{\frac{1}{6}+2\eta}<d_2p_j$,
we deduce $e_2 := d_2p_j$ lies in the interval $e_2 \in [x^{\frac{1}{6}+2\eta},x^{\frac{1}{4}-3\eta}]$.

Thus letting $E_3 := D_2D_3/e_2$, for each $l>j$ in turn we shall greedily append the prime $p_l$ onto either $d_1$ or $d_3$ while preserving $d_1<D_1$ and $d_3<E_3$. Indeed, for all $l>j$,
\begin{align*}
p_l^2 \le \frac{D}{p_1\cdots p_{l-1}} = \frac{D_1D_2D_3}{d_1d_2d_3p_j\cdots p_{l-1}} \le \frac{D_1E_3}{d_1d_3p_{j+1}\cdots p_{l-1}},
\end{align*}
so there is a factorization $e_1e_3 = d_1d_3p_{j+1}\cdots p_l$ with $e_1\le D_1=N$ and $e_3\le E_3= D/(Ae_2x^{2\delta})$. Hence the result now follows by Lemma \ref{lem:criterion} for the factorization $e_1e_2e_3=d_1d_2d_3p_{j}\cdots p_l = p_1\cdots p_l$.
\end{proof}

\begin{proof}[Proof of Proposition \ref{prop:factorD}]
We shall consider 3 cases, depending on the sizes of $p_1$ and $p_2p_3$ compared to the endpoints of the key interval $[x^{\frac{1}{6}+2\eta},x^{\frac{1}{4}-3\eta}]$.

\vspace{1em}
\noindent
{\bf CASE 1: $p_1 \ge D^2/x=x^{\frac{1}{6}+2\eta}$.}
\vspace{.5em}

Let $d_2 := p_1$, $C:= D/Nd_2$. Note $C = D/Nd_2 \ge D^{\frac{2}{3}}/N \ge 1$.

Next $D\ge p_1^3\ge p_1p_2^2$ implies $p_2^2 \le D/p_1=NC$, so either $p_2 \le N$ or $p_2\le C$. Similarly, since $p_1\cdots p_{j-1}p_j^2\le D$ for all $j\le r$, we get $p_j^2 \le \frac{AC}{p_2\cdots p_{j-1}}$ for $3\le j\le r$. As such, we may factor $p_2\cdots p_r=d_1d_3$ for $d_1\le N, d_3\le C$. Hence by Lemma \ref{lem:criterion} $p_1\cdots p_r=d_1d_2d_3$ satisfies \eqref{eq:JMfactorable}.

\vspace{.5em}

In the remaining cases, we assume $p_1<x^{\frac{1}{6}+2\eta}$. By Lemma \ref{lem:reduce}, it remains to consider $p_2p_3>x^{\frac{1}{4}-3\eta}$ or $p_2p_3< x^{\frac{1}{6}+2\eta}$. Note
\begin{align}\label{eq:p13}
p_2p_3 \le p_1^{\frac{1}{3}}(p_1p_2p_3^3)^{\frac{1}{3}} \le (x^{\frac{1}{6}+2\eta})^{\frac{1}{3}}D^{\frac{1}{3}} = x^{\frac{1}{3}(\frac{1}{6}+\frac{7}{12}+3\eta-50\eta)} < x^{\frac{1}{4}+\eta-16\delta}.
\end{align}

\vspace{1em}
\noindent
{\bf CASE 2: $p_2p_3>x^{\frac{1}{4}-3\eta}$ and $p_1 <x^{\frac{1}{6}+2\eta}$.}
\vspace{.5em}

The proof follows by Lemma \ref{lem:reduce} if $p_2p_4\in[x^{\frac{1}{6}+2\eta},x^{\frac{1}{4}-3\eta}]$. Thus by definition of $\mathcal P_4$, in this case we may assume
\begin{align}
p_2p_4<x^{\frac{1}{6}+2\eta}.
\end{align}
Hence we have $p_4<x^{12\eta+50\delta}$, since
\begin{align}\label{eq:p2X1}
p_2 > p_2(p_1p_2p_3^3)/D > (p_2p_3)^3/D > x^{3(\frac{1}{4}-3\eta)}/D = x^{\frac{1}{6}-10\eta+50\delta}.
\end{align}

If $p_1p_4>x^{\frac{1}{6}+2\eta}$, then the proof follows by Lemma \ref{lem:reduce} where $d_2=p_1p_4$ is $<x^{(\frac{1}{6}+2\eta)+12\eta+50\delta}<x^{\frac{1}{4}-3\eta}$, since $\eta<\frac{1}{204}-3\delta$.

Else $p_1p_4<x^{\frac{1}{6}+2\eta}$. We shall apply Lemma \ref{lem:greedy} with $d_2=p_1p_4$.

If either $Nx^{-\delta}$ or $x^{1-2\delta}/DN$ is greater than $x^{\frac{1}{4}+\eta-16\delta} \ge p_2p_3$, by \eqref{eq:p13}, then Lemma \ref{lem:greedy} completes the proof with $(d_1,d_3)=(p_2p_3,1)$ or $(1,p_2p_3)$, respectively. Otherwise, $Nx^{-\delta}, x^{1-2\delta}/DN\in  [x^{\frac{1}{6}-2\eta-64\delta},x^{\frac{1}{4}+\eta-16\delta}]$, since $x/D=x^{\frac{5}{12}-\eta+50\delta}$. But then, using $\eta<\frac{1}{108}$,
\begin{align}\label{eq:Aminmax}
\max(Nx^{-\delta},x^{1-2\delta}/DN)\ge\; (x^{1-2\delta}/D)^{\frac{1}{2}} = x^{\frac{5/2}{12}-\eta/2+24\delta} > x^{\frac{1}{6}+2\eta} > p_2, \nonumber\\
\min(Nx^{-\delta},x^{1-2\delta}/DN)\ge\; x^{\frac{1}{6}-2\eta-64\delta} >  x^{\frac{1}{8}+\eta/2-8\delta} \ge (p_2p_3)^{1/2} \ge p_3,
\end{align}
by \eqref{eq:p13}, which suffices again for Lemma \ref{lem:greedy}. Note $p_6<x^{12\eta+50\delta}<x^{\frac{1}{12}-5\eta}$ when $r\ge6$, using $\eta<\frac{1}{204}-3\delta$.

\vspace{1em}
\noindent
{\bf CASE 3: $p_2p_3< x^{\frac{1}{6}+2\eta}$ and $p_1 <x^{\frac{1}{6}+2\eta}$.}
\vspace{.5em}

By Lemma \ref{lem:reduce}, it suffices to consider either $p_1p_4 < x^{\frac{1}{6}+2\eta}$ or $p_1p_4 > x^{\frac{1}{4}-3\eta}$.

\vspace{1em}
{\bf Subcase 3.1: $p_1p_4 < x^{\frac{1}{6}+2\eta}$.}
\vspace{.5em}

Suppose we can show $p_6 < x^{\frac{1}{12}-5\eta}$ (when $r\ge6$). Then since $x^{1-3\delta}/D > D^4/x^{2-6\delta}$, either $Nx^{-\delta}$ or $x^{1-2\delta}/DN$ is greater than $D^2/x^{1-3\delta}$. Thus Lemma \ref{lem:greedy} will complete the proof, with $(d_1,d_2,d_3)=(p_1p_4,p_2p_3,1)$ or $(1,p_2p_3,p_1p_4)$.

If $p_1p_3 > x^{\frac{1}{4}-3\eta}$ then in this subcase
\begin{align*}
x^{\frac{1}{12}+\eta} > (p_2p_3)^{\frac{1}{2}} \ge p_3 = p_4\frac{p_1p_3}{p_1p_4}>
p_4x^{\frac{1}{12}-5\eta},
\end{align*}
so $p_4<x^{6\eta}$. Hence $p_6\le p_4 < x^{\frac{1}{12}-5\eta}$ since $\eta < \frac{1}{108}$, which completes the proof.

Else $p_1p_3 < x^{\frac{1}{4}-3\eta}$. By Lemma \ref{lem:reduce}, it suffices $p_1p_3 < x^{\frac{1}{6}+2\eta}$. Then we see $p_3 > x^{\frac{1}{12}-5\eta}$ implies $p_1 < x^{\frac{1}{12}+7\eta}$. If further $p_1p_2 > x^{\frac{1}{4}-3\eta}$, then similarly 
\begin{align*}
x^{\frac{1}{12}+7\eta} > p_1 \ge p_2 = p_3\frac{p_1p_2}{p_1p_3}>
p_3x^{\frac{1}{12}-5\eta},
\end{align*}
so $p_3<x^{12\eta}$. Hence $p_6\le p_3 < x^{\frac{1}{12}-5\eta}$ since $\eta < \frac{1}{204}$, which completes the proof.

Else $p_1p_2 < x^{\frac{1}{4}-3\eta}$. By Lemma \ref{lem:reduce}, we may assume $p_1p_2 < x^{\frac{1}{6}+2\eta}$.

Similarly, suppose $p_2p_3p_4<x^{\frac{1}{4}-3\eta}$. By Lemma \ref{lem:reduce} we may assume $p_2p_3p_4 < x^{\frac{1}{6}+2\eta}$, and so
\begin{align*}
p_6\le (p_2p_3p_4)^{\frac{1}{3}} \le x^{\frac{(2/3)}{12}+(2/3)\eta} \le x^{\frac{1}{12}-5\eta},
\end{align*}
using $\eta<\frac{1}{204}$, which completes the proof.

Thus we may assume $p_2p_3p_4 > x^{\frac{1}{4}-3\eta}$. But unless $p_6<x^{\frac{1}{12}-5\eta}$, this subcase will contradict the definition of $\mathcal P_{6,1}$ in \eqref{eq:defP4P6}, hence completing the proof.

\vspace{1em}
{\bf Subcase 3.2: $p_1p_4 > x^{\frac{1}{4}-3\eta}$.}
\vspace{.5em}

If $d_2=p_2p_3p_4<x^{\frac{1}{6}+2\eta}$, then Lemma \ref{lem:greedy} completes the proof with $(d_1,d_3)=(p_1,1)$ or $(1,p_1)$, since
\begin{align*}
p_6 \le (p_2p_3p_4)^{\frac{1}{3}} \le x^{\frac{1}{3}(\frac{1}{6}+2\eta)} \le x^{\frac{1}{12}-5\eta}
\end{align*}
for $\eta<\frac{1}{204}$. And if $p_2p_3p_4\in [x^{\frac{1}{6}+2\eta},x^{\frac{1}{4}-3\eta}]$ the proof follows by Lemma \ref{lem:reduce}.

Else $p_2p_3p_4>x^{\frac{1}{4}-3\eta}$. Note $p_4<x^{\frac{1}{12}+\eta}$ and $p_1 = p_1p_4/p_4 > x^{\frac{1}{6}-4\eta}$ and $p_2p_3p_4<x^{\frac{1}{4}+3\eta}$. Also note we may factor $p_1p_4=d_1d_3$ for $d_1\le N$, $d_3\le x^{1-2\delta}/DN$ (Indeed this follows if $N$ or $x^{1-2\delta}/DN$ exceeds $x^{\frac{1}{4}+3\eta}\ge p_1p_4$. Else $N,x^{1-2\delta}/DN\in [x^{\frac{1}{6}-4\eta-2\delta},x^{\frac{1}{4}+3\eta}]$, which also works similarly as with \eqref{eq:Aminmax}, since $p_4<x^{\frac{1}{12}+\eta}<x^{\frac{1}{6}-4\eta-2\delta}$ and $p_1<x^{\frac{1}{6}+2\eta}<x^{\frac{1}{2}(\frac{5}{12}-\eta-\delta)}$ by $\eta<\frac{1}{60}$)

If further $p_6>x^{\frac{1}{12}-5\eta}$, then this subcase contradicts the definition of $\mathcal P_{6,2}$ in \eqref{eq:defP4P6}. Hence we have $p_6\le x^{\frac{1}{12}-5\eta}$, and so by the above paragraph Lemma \ref{lem:greedy} completes the proof with $d_2=p_2p_3$.

\vspace{.5em}

Combining all cases completes the proof of Proposition \ref{prop:factorD}.
\end{proof}

\subsection{Refined factorization of $\mathcal D^{\textnormal{well}}$}

Proposition \ref{prop:factorD} (programmably) factorizes each $d\in\mathcal D^* \subset \mathcal D^+(x^{\frac{7}{12}+\eta})$, and forms the key step to prove the weights $\lambda^*$ are programmably factorable. With applications in mind to twin primes, we shall similarly (programmably) factorize certain subsets of the well-factorable support $\mathcal D^{\textnormal{well}}$, as in \eqref{eq:defDwell}.

In the following result, we factorize $d\in\mathcal D^{\textnormal{well}}(D)$ for variable level $D\in(x^{\frac{4}{7}},x^{\frac{3}{5}})$, depending on the anatomy of $d$. As $\mathcal D^\pm \subset \mathcal D^{\textnormal{well}}$, this has implications to both upper and lower bounds for the standard linear sieve.

\begin{proposition}
Let $\mathcal D^{\textnormal{well}}(D)$ as in \eqref{eq:defDwell} for $D=x^{\frac{7}{12}+\eta-50\delta}$ and $-\frac{1}{84}<\eta<\frac{1}{60}-30\delta$. Let $x^{\frac{1}{4}-3\eta} \ge p_1\ge\cdots\ge p_r$ be primes for which $d=p_1\cdots p_r\in \mathcal D^{\textnormal{well}}(D)$. Then $d$ has factorization $d=abc$ satisfying \eqref{eq:JMfactorable}  if $p_3 \ \le \ x^{\frac{1}{12}-5\eta}$, or if
\begin{align*}
d_2\in [x^{\frac{1}{6}+2\eta},x^{\frac{1}{4}-3\eta}] \qquad \textnormal{with } \ d_2\mid p_1p_2p_3, d_2\neq p_3.
\end{align*}
\end{proposition}
\begin{proof}
For $i=1,2,3$, suppose $d_2=p_1\cdots p_i$ lies $[x^{\frac{1}{6}+2\eta},x^{\frac{1}{4}-3\eta}]$, and let $A=Nx^{-\delta}$, $C=x^{\delta}D/Nd_2$. Since $p_1\cdots p_{j-1}p_j^2\le D$ for all $i<j\le r$, we get $p_j^2 \le \frac{AC}{p_{i+1}\cdots p_{j-1}}$ for $i< j\le r$. As such, we may factor $p_{i+1}\cdots p_r=d_1d_3$ for $d_1\le A, d_3\le C$. Hence by Lemma \ref{lem:criterion} $p_1\cdots p_r=d_1d_2d_3$ satisfies \eqref{eq:JMfactorable}.

Else, by assumption $p_1<x^{\frac{1}{4}-3\eta}$ so we may assume further $p_1<x^{\frac{1}{6}+2\eta}$. In particular this gives $p_1^2\le D/d_2$. For the remaining $d_2\mid p_1p_2p_3$:
\begin{itemize}
\item If $d_2=p_2p_3$ then $p_1^2 \le D/d_2=AC$ implies $p_1\le A$ or $p_1\le C$.
\item If $d_2=p_1p_3$ then $p_2^2 \le D/d_2=AC$ implies $p_2\le A$ or $p_2\le C$.
\item If $d_2=p_2$ then $p_3^2 \le D/p_1d_2$ implies a factorization $p_1p_3=d_1d_3$ for $d_1\le A, d_3\le C$.
\end{itemize}
For each $d_2$ above, we factored $p_1p_2p_3=d_1d_2d_3$ for $d_1\le A, d_3\le C$. Since $p_1\cdots p_{j-1}p_j^2\le D$ for all $j\le r$, by induction we may factor $p_1\cdots p_r=d_1d_2d_3$ for $d_1\le A, d_3\le C$. By Lemma \ref{lem:criterion} $p_1\cdots p_r=d_1d_2d_3$ satisfies \eqref{eq:JMfactorable}.

Finally, suppose $p_3 \ \le \ x^{\frac{1}{12}-5\eta}$ is less than the width of the interval $[x^{\frac{1}{6}+2\eta},x^{\frac{1}{4}-3\eta}]$. Since $p_1<x^{\frac{1}{6}+2\eta}$, we have $p_1p_3< x^{\frac{1}{4}-3\eta}$ so by the above argument we may assume $d_2:=p_1p_3< x^{\frac{1}{6}+2\eta}$. Then $p_2^3\le p_1p_2^2 \le D$ implies
\begin{align*}
p_2^2\le x^{\frac{2}{3}(\frac{7}{12}+\eta)} < x^{\frac{5}{12}-\eta+47\delta}= \frac{x^{1-3\delta}}{D},
\end{align*}
since $\eta<\frac{1}{60}-30\delta$. Thus $p_2\le Nx^{-\delta}$ or $p_2\le x^{1-2\delta}/DN$, so there is a factorization $p_1p_2p_3=d_1d_2d_3$ satisfying \eqref{eq:greedy}. Hence the same greedy argument as in Lemma \ref{lem:greedy} completes the proof, with $p_3$ playing the role of $p_6$.
\end{proof}

Taking the maximum valid $\eta$ as above, we may re-express the above factorization of level $x^{\theta}$, $\theta=\frac{7}{12}+\eta$, as follows. Note the maximum $\theta$ for which $t\in [\frac{1}{6}+2\eta, \frac{1}{4}-3\eta]$ is given by
\begin{align}\label{eq:thetat}
\theta(t) = \begin{cases}
\frac{2-t}{3} & \text{if} \ \ t > \frac{1}{5},\\
\frac{1+t}{2} & \text{if} \ \ t \le \, \frac{1}{5}.
\end{cases}
\end{align}
Similarly the maximum $\theta=\frac{7}{12}+\eta$ for which $t\le \frac{1}{12}-5\eta$ is $(3-t)/5$.

\begin{corollary}\label{cor:piecewiseeta}
Let $p_1\ge\cdots\ge p_r$ be primes and write $p_i=x^{t_i}$. If $d=p_1\cdots p_r\in\mathcal D^{\textnormal{well}}(x^{\theta-50\delta})$, then there is a factorization $d=d_1d_2d_3$ satisfying \eqref{eq:JMfactorable} provided
\begin{align*}
\theta \ \le \ \theta(t_1),
\end{align*}
for $\theta(t)$ as in \eqref{eq:thetat}. Moreover if $t_1 \le \,\frac{1}{5}$, then it suffices that
\begin{align}\label{eq:thetat123}
\theta \ \le \ \theta(t_1,t_2,t_3) := \max\Big\{ \, &\frac{3-t_3}{5},\,\theta(t_1),\,\theta(t_2),\,\theta(t_1+t_2+t_3),\\
& \ \ \theta(t_1+t_2),\,\theta(t_1+t_3), \,\theta(t_2+t_3)\Big\}. \nonumber
\end{align}
\end{corollary}

\section{Modification of the linear sieve}

In this section we shall bound the modified linear sieve, analogous to the bounds for the linear sieve (sometimes called the Jurkat--Richert theorem). This bound will form the basis for our final result in the next section, in which we modify the construction of Iwaniec's weights.

\begin{proposition}\label{prop:modJurkatRich}
Let $\eps>0$ be sufficiently small. For $\eta\le \frac{1}{204}$, the modified weights $\lambda^*$ as in \eqref{eq:lambstardef} of level $D=x^{\frac{7}{12}+\eta-\eps}$ satisfy
\begin{align*}
S(\mathcal A,z) \ \le \ |\mathcal A|V(z)\Big(F^*(\tfrac{\log D}{\log z}) + o(1)\Big) + \sum_{d\mid P(z)}\lambda^*(d) \,r_\mathcal{A}(d),
\end{align*}
where $F^*=F^*_\eta$ is a function satisfying $F^*(s) = F(s) + O(\eta^5)$ for $F$ as in \eqref{eq:delaydiff}.
\end{proposition}
\begin{remark}
It suffices for our purposes to obtain qualitative error $o(1)$ in the factor accompanying $F^*$. Though as with the Jurkat--Richert theorem, with greater care one should obtain a quantitative refinement, e.g. $O((\log D)^{-1/6})$. See (12.4)--(12.8) in \cite{Opera}.
\end{remark}

We now adapt the proof. Let $D=x^{\frac{7}{12}+\eta}$ and $D_0 = x^{\frac{7}{12}}$. For $n\ge1$, primes $p_1\ge\cdots\ge p_n$, if $p_1\cdots p_n\notin \mathcal D^+(D)$ then there exists a minimal index $l\le n$ such that $p_1\cdots p_l\notin \mathcal D^+(D)$. By definition such minimal $l$ is odd. (Explicitly, this occurs when $p_1\cdots p_{l-1}p_l^3 > D$ but $p_1\cdots p_{m-1}p_m^3 \le D$ for all odd $m<l$.)
Similarly, if $p_1\cdots p_n\notin \mathcal D^*$ there exists a minimal index $l\le n$ such that $p_1\cdots p_l\notin \mathcal D^*$, which is also odd.

Indeed, to show this let $l\le n$ be the minimal index such that $p_1\cdots p_l\notin \mathcal D^*$. If $(p_1,\ldots, p_j)\notin\mathcal P_j$ for all $j\le l$, $j\in\{4,6\}$, then clearly $l>j$ must be odd, as with $\mathcal D^+(D)$. On the other hand, if $(p_1,\ldots, p_j)\in\mathcal P_j$ for some $j\le l$, $j\in\{4,6\}$, a priori one might expect $l$ could be even.
However, the key point in this case is that $p_1\cdots p_j\in \mathcal D^+(D_0)\subset \mathcal D^*$ (since $p_1\cdots p_j\approx D^{\frac{6}{7}}$ by definition of $\mathcal P_j$). Thus $l>j$ is the minimal index such that $p_1\cdots p_l\notin \mathcal D^+(D_0)$, and hence must be odd as claimed.

Using this minimal index, we show the following lemma.
\begin{lemma}\label{lem:mutolambda}
Let $h$ be a multiplicative function with $0\le h(p)\le 1$ for all primes $p$. Then we have
\begin{align*}
\prod_{p\mid n}(1-h(p)) \ \le \ \sum_{d\mid n}\lambda^*(d) h(d).
\end{align*}
\end{lemma}
\begin{proof}
Note if $h=1$ identically, we interpret the product as $\1_{n=1}$. Now by definition,
\begin{align*}
\sum_{d\mid n}\lambda^*(d) h(d)  -  \prod_{p\mid n}(1-h(p)) \ = \ 
\sum_{\substack{d\mid n\\d\in\mathcal D^*}}\mu(d) h(d) - \sum_{d\mid n}\mu(d) h(d)
\ = \ -\sum_{\substack{d\mid n\\d\notin\mathcal D^*}}\mu(d) h(d).
\end{align*}
Then splitting up $d\notin \mathcal D^*$ by its minimal index,
\begin{align*}
-\sum_{\substack{d\mid n\\d\notin\mathcal D^*}}\mu(d) h(d) & = \sum_{\text{odd }l}\sum_{\substack{p_l<\cdots< p_1<z\\p_1\cdots p_{l-1}\in \mathcal D^*\\ p_1\cdots p_l\notin \mathcal D^*}}h(p_1\cdots p_l)\sum_{\substack{p_1\cdots p_l b\mid n\\ b \mid P(p_l)}}\mu(b) h(b) \ \ge \ 0,
\end{align*}
since $h\ge0$ and the inner sum over $b$ factors as $\prod_{p\mid (P(p_l), n)}(1-h(p)) \ge 0$, since $h(p)\le 1$.
\end{proof}

By Lemma \ref{lem:mutolambda} with $h(d)=1$, we have
\begin{align}\label{eq:Mobius}
\1_{n=1} \ \le \ 
\sum_{d\mid n}\lambda^*(d),
\end{align}
in which case we obtain
\begin{align}\label{eq:lambinvers}
S(\mathcal{A},\mathcal P,z) &= \sum_{n\in A}\1_{(n,P(z))=1} \  \le \ \sum_{n\in A}\sum_{d\mid (n,P(z))}\lambda^*(d) = \sum_{d\mid P(z)}\lambda^*(d)|\mathcal{A}_d| \nonumber\\
& \qquad = X\sum_{d\mid P(z)}\lambda^*(d) \,g(d) + \sum_{d\mid P(z)}\lambda^*(d) \,r_\mathcal{A}(d) \ =: \ XV^*(D,z) + R^*_\mathcal{A}(D,z).
\end{align}

Following Lemma \ref{lem:mutolambda} with $h=g$, we have the identity
\begin{align}\label{eq:Vstardef}
V^*(D,z) & := \sum_{\substack{d\mid P(z)\\d\in\mathcal D^*}}\mu(d)g(d)  = V(z) + \sum_{\text{odd }n}\sum_{\substack{p_n<\cdots< p_1<z\\p_1\cdots p_{n-1}\in \mathcal D^*\\ p_1\cdots p_n\notin \mathcal D^*}}g(p_1\cdots p_n)V(p_n),
\end{align}
and similarly
\begin{align}\label{eq:V+Vn}
V^+(D,z) & = V(z) + \sum_{\text{odd }n}\sum_{\substack{p_n<\cdots< p_1<z\\p_1\cdots p_{n-1}\in \mathcal D^+(D)\\ p_1\cdots p_n\notin \mathcal D^+(D)}}g(p_1\cdots p_n)V(p_n) \ =: \ V(z) + \sum_{\text{odd }n}V_n(D,z).
\end{align}
Then the difference of $V^*$ and $V^+$ is
\begin{align}\label{eq:VplusVstar1}
V^*(D,z) \ - \ & V^+(D,z)  \ = \ \sum_{\text{odd }n}\sum_{p_n<\cdots< p_1<z}g(p_1\cdots p_n)V(p_n)\,{\bf \Delta},
\end{align}
where ${\bf \Delta}$ is the difference of indicator functions,
\begin{align*}
{\bf \Delta} : = \ \1_{\substack{p_1\cdots p_n\notin \mathcal D^*\\p_1\cdots p_{n-1}\in \mathcal D^*}} \ - \ \1_{\substack{p_1\cdots p_n\notin \mathcal D^+(D)\\p_1\cdots p_{n-1}\in \mathcal D^+(D)}} & \ = \ \1_{\substack{p_1\cdots p_n\in \mathcal D^+(D)\setminus \mathcal D^*\\p_1\cdots p_{n-1}\in \mathcal D^*}} \ - \ \1_{\substack{p_1\cdots p_n\notin \mathcal D^+(D)\\p_1\cdots p_{n-1}\in \mathcal D^+(D)\setminus \mathcal D^*}},
\end{align*}
recalling $\mathcal D^*\subset \mathcal D^+(D)$. Note if a point is $(p_1,..,p_6)\in\mathcal P_6$ then its projection is $(p_1,..,p_4)\notin\mathcal P_4$. So by definitions of $\mathcal D^*$, $\mathcal D^+(D)$ from \eqref{eq:Dstardef}, \eqref{eq:Dplusdef}, for odd $n$ we have the identities,
\begin{align}
\1_{\substack{p_1\cdots p_n\in \mathcal D^+(D)\setminus \mathcal D^*\\p_1\cdots p_{n-1}\in \mathcal D^*}} & = 
\sum_{\substack{j\in\{4,6\}\\j<n}}\1_{(p_1,..,p_j)\in \mathcal P_j}\cdot \1_{\substack{p_1\cdots p_n\in \mathcal D^+(D)\setminus\mathcal D^+(D_0)\\p_1\cdots p_{n-1}\in \mathcal D^+(D_0)}}, \label{eq:Dnotstarstar}\\
\1_{\substack{p_1\cdots p_n\notin \mathcal D^+(D)\\p_1\cdots p_{n-1}\in \mathcal D^+(D)\setminus \mathcal D^*}} &=
\sum_{\substack{j\in\{4,6\}\\j<n}}\1_{(p_1,..,p_j)\in \mathcal P_j}\cdot \1_{\substack{p_1\cdots p_n\notin \mathcal D^+(D)\\p_1\cdots p_{n-1}\in \mathcal D^+(D)\setminus \mathcal D^+(D_0)}}. \label{eq:DDnotstar}
\end{align}
We may plug \eqref{eq:Dnotstarstar} and \eqref{eq:DDnotstar} into ${\bf \Delta}$. In addition, we strategically add and subtract the indicator function of $\{p_1\cdots p_n\notin \mathcal D^+(D)$,\, $p_1\cdots p_{n-1}\in \mathcal D^+(D_0)\}$, which together give
\begin{align*}
{\bf \Delta}  & \ = \ \sum_{\substack{j\in\{4,6\}\\j<n}}\1_{(p_1,..,p_j)\in \mathcal P_j} \cdot\bigg(\1_{\substack{p_1\cdots p_n\in \mathcal D^+(D)\setminus\mathcal D^+(D_0)\\p_1\cdots p_{n-1}\in \mathcal D^+(D_0)}}  - \1_{\substack{p_1\cdots p_n\notin \mathcal D^+(D)\\p_1\cdots p_{n-1}\in \mathcal D^+(D)\setminus \mathcal D^+(D_0)}}\bigg)\\
& = \sum_{\substack{j\in\{4,6\}\\j<n}}\1_{(p_1,..,p_j)\in \mathcal P_j}
\cdot \bigg(\1_{\substack{p_1\cdots p_n\in \mathcal D^+(D)\setminus\mathcal D^+(D_0)\\p_1\cdots p_{n-1}\in \mathcal D^+(D_0)}} + \1_{\substack{p_1\cdots p_n\notin \mathcal D^+(D)\\p_1\cdots p_{n-1}\in \mathcal D^+(D_0)}}\\
& \qquad\qquad\qquad\qquad\qquad\qquad - \1_{\substack{p_1\cdots p_n\notin \mathcal D^+(D)\\p_1\cdots p_{n-1}\in \mathcal D^+(D)\setminus \mathcal D^+(D_0)}} - \1_{\substack{p_1\cdots p_n\notin \mathcal D^+(D)\\p_1\cdots p_{n-1}\in\mathcal D^+(D_0)}}\bigg)\\
& \ = \ \sum_{\substack{j\in\{4,6\}\\j<n}}\1_{(p_1,..,p_j)\in \mathcal P_j}\cdot\bigg(\1_{\substack{p_1\cdots p_n\notin \mathcal D^+(D_0)\\p_1\cdots p_{n-1}\in \mathcal D^+(D_0)}}  - \1_{\substack{p_1\cdots p_n\notin \mathcal D^+(D)\\p_1\cdots p_{n-1}\in \mathcal D^+(D)}}\bigg).
\end{align*}

Thus plugging ${\bf \Delta}$ back into \eqref{eq:VplusVstar1} gives
\begin{align*}
V^* (D,z)  \ - \ & V^+(D,z)  \ = \
\sum_{j\in \{4,6\}}  \sum_{\substack{p_j<\cdots< p_1<z\\(p_1,..,p_j)\in\mathcal P_j}}
\sum_{\text{odd }n>j}\times\\
&\times\sum_{p_n<\cdots p_{j+1}<p_j}g(p_1\cdots p_n)V(p_n) \bigg( \1_{\substack{p_1\cdots p_n\notin \mathcal D^+(D_0)\\p_1\cdots p_{n-1}\in \mathcal D^+(D_0)}}  - \1_{\substack{p_1\cdots p_n\notin \mathcal D^+(D)\\p_1\cdots p_{n-1}\in \mathcal D^+(D)}}\bigg).
\end{align*}
Recalling the definition of $V_n(D,z)$ in \eqref{eq:V+Vn}, since $g$ is multiplicative we have
\begin{align}\label{eq:VplusVstar2}
V^*(D,z)  &-  V^+(D,z) \nonumber\\
&= \ \sum_{j\in \{4,6\}}  \sum_{\substack{p_j<\cdots< p_1<z\\(p_1,..,p_j)\in\mathcal P_j}} g(p_1\cdots p_j)\sum_{\text{odd }n>j}\Big(V_{n-j}\big(\tfrac{D_0}{p_1\cdots p_j},p_j\big) - V_{n-j}\big(\tfrac{D}{p_1\cdots p_j},p_j\big)\Big) \nonumber\\
& = \sum_{j\in \{4,6\}} \sum_{\substack{p_j<\cdots< p_1<z\\(p_1,..,p_j)\in\mathcal P_j}} g(p_1\cdots p_j) \Big(V^+\big(\tfrac{D_0}{p_1\cdots p_j},p_j\big) - V^+\big(\tfrac{D}{p_1\cdots p_j},p_j\big)\Big).
\end{align}
as $V^+(D,z) - V^+(D',z)=\sum_{\text{odd }n}[V_n(D,z)-V_n(D',z)]$.

Now assuming the two-sided condition \eqref{eq:Vwz} for $g$, the proof of \cite[Theorem 11.12]{Opera} (c.f. (12.4)--(12.8)) gives asymptotic equality,
\begin{align}\label{eq:Vplus}
V^+(D,z) \ &= \ V(z)\Big\{ F(\tfrac{\log D}{\log z}) \ + \ O\big((\log D)^{-1/6}\big)\Big\} \qquad (z\le D), 
\end{align}
so that \eqref{eq:VplusVstar2} becomes
\begin{align}\label{eq:VplusVstar3}
 V^*(D,z)  \ & = \ V(z)\Big\{ F(\tfrac{\log D}{\log z} ) \ + \ O\big((\log D)^{-\frac{1}{6}}\big)\Big\}  \nonumber\\
 & + \sum_{j\in \{4,6\}} \sum_{\substack{p_j<\cdots< p_1<z\\(p_1,..,p_j)\in\mathcal P_j}} g(p_1\cdots p_j) V(p_j)\times\\
&\qquad\quad\times\Big\{ F(\tfrac{\log D_0/p_1\cdots p_j}{\log p_j}) - F(\tfrac{\log D/p_1\cdots p_j}{\log p_j}) \ + \ O\big(\log\big(\tfrac{D}{p_1\cdots p_j}\big)^{-\frac{1}{6}}\big)\Big\} \nonumber
\end{align}
By partial summation and the prime number theorem, for each $j$ we have,
\begin{align*}
\sum_{\substack{p_j<\cdots< p_1<z\\(p_1,..,p_j)\in\mathcal P_j}} g(p_1 &\cdots p_j)  V(p_j)F(\tfrac{\log D_0/p_1\cdots p_j}{\log p_j}) \\
&= (\tfrac{7}{12}+\eta)
\int_{(x_1,..,x_j)\in P_j}\frac{\dd{x_1}\cdots \dd{x_j}}{x_1\cdots x_{j-1} x_j^2}
F\Big(\tfrac{\frac{7}{12}-x_1-\cdots x_j}{x_j}\Big) + O\big((\log D)^{-\frac{1}{6}}\big).
\end{align*}
Here $P_j$ is the polytope in Euclidean space $\R^j$ corresponding to $\mathcal P_j$, as below.

Hence from \eqref{eq:VplusVstar3}, we obtain
\begin{align}\label{eq:VFstar}
V^*(D,z) = V(z)\Big\{ F^*(\tfrac{\log D}{\log z}) \ + \ O\big((\log D)^{-\frac{1}{6}}\big)\Big\}  \qquad (z\le D),
\end{align}
where the function $F^*$ satisfies
\begin{align}\label{eq:FFstar1}
sF^*(s) \ - \ sF(&s) \ = \\
(\tfrac{7}{12}+\eta) &\cdot  \sum_{j\in\{4,6\}}\int_{(x_1,..,x_j)\in P_j}\frac{\dd{x_1}\cdots \dd{x_j}}{x_1\cdots x_{j-1} x_j^2}
\bigg[F\Big(\tfrac{\frac{7}{12}-x_1-\cdots x_j}{x_j}\Big) - 
F\Big(\tfrac{\frac{7}{12}+\eta-x_1-\cdots x_j}{x_j}\Big)\bigg]. \nonumber
\end{align}
Namely, $P_4\subset \R^4$ is given by
\begin{align*}
P_4 = \{(x_1,...,x_4)\in\mathrm D^+(\tfrac{7}{12}+\eta) \; &: \;  x_1 < \tfrac{1}{6}+2\eta \ \text{ and } \ x_2+x_4>\tfrac{1}{4}-3\eta\}, \end{align*}
and $P_6=P_{6,1}\cup P_{6,2} \subset \R^6$ is given by
\begin{align}
P_{6,1} = \{(x_1,...,x_6)\in\mathrm D^+(\tfrac{7}{12}+\eta) \; &: \; x_1+x_2 < \tfrac{1}{6}+2\eta \ \text{ and } \ x_6>\tfrac{1}{12}-5\eta  \\
&\qquad\text{ and } \ x_2+x_3+x_4>\tfrac{1}{4}-3\eta\}, \nonumber\\
P_{6,2} = \{(x_1,...,x_6)\in\mathrm D^+(\tfrac{7}{12}+\eta) \; &: \; x_1,\;x_2+x_3 < \tfrac{1}{6}+2\eta \ \text{ and } \ x_6>\tfrac{1}{12}-5\eta \nonumber\\
&\qquad\text{ and } \ x_1+x_4,\; x_2+x_3+x_4 > \tfrac{1}{4}-3\eta \}, \nonumber
\end{align}
Similarly, $\mathrm D^+$ is the set in Euclidean space corresponding to $\mathcal D^+$, namely,
\begin{align*}
\mathrm D^+(\tau) =  \{ (x_1,\ldots, x_r) \;  : &\; x_1 > \cdots > x_r > 0,\\
& \text{ and }
x_1+\cdots x_{l-1}+3x_l < \tau \ \text{ for each odd }1\le l \le r\}.
\end{align*}
Hence Proposition \ref{prop:modJurkatRich} follows.

\subsection{Sieve function computation}

We now compute $F^*$ in terms of $F$.
\begin{proposition}
Let $\eta = \frac{1}{204}$. Then for $1\le\; s\le\; 3$, we have
\begin{align}\label{eq:triangleFstar}
F^*(s) \ \le \ 1.000081\,F(s).
\end{align}
\end{proposition}
\begin{proof}
From \eqref{eq:FFstar1} we have
\begin{align}\label{eq:FFstar2}
sF^*(s) = sF(s) + \big(\tfrac{7}{12}+\eta\big)\cdot 2e^\gamma \eta(J_4+J_6),
\end{align}
for integrals $J_j$, $j\in\{4,6\}$,
\begin{align*}
J_j \ & := \ \frac{1}{2e^\gamma \eta}\int_{(x_1,..,x_j)\in P_j}\frac{\dd{x_1}\cdots \dd{x_j}}{x_1\cdots x_{j-1} x_j^2}
\bigg[F\Big(\tfrac{\frac{7}{12}-x_1-\cdots x_j}{x_j}\Big) - 
F\Big(\tfrac{\frac{7}{12}+\eta-x_1-\cdots x_j}{x_j}\Big)\bigg]\\
\ &= \ \int_{(x_1,..,x_j)\in P_j}\frac{\dd{x_1}\cdots \dd{x_j}}{x_1\cdots x_j}\Big[(\tfrac{7}{12}-x_1-\cdots x_j)(\tfrac{7}{12}+\eta-x_1-\cdots x_j)\Big]^{-1},
\end{align*}
since $sF(s)=2e^\gamma$ for $s\in[1,3]$. In particular $|P_j|=O(\eta^j)$ implies $J_j = O(\eta^{j})$, and so from \eqref{eq:FFstar2} we obtain $F^*(s) = F(s) + O(\eta^5)$.

For $\eta = \frac{1}{204}$, we use Mathematica\footnote{The Mathematica package and code are available at {\tt arxiv.org/abs/2109.02851}} to compute that 
\begin{align}\label{eq:I4}
J_4 \ \le \ 0.016896.
\end{align}

Next we bound $J_6$. For $(x_1,..,x_6)\in P_6$ we have $x_4<\frac{1}{2}(x_2+x_3)<\frac{1}{12}+\eta$ and $\frac{7}{12}+\eta-x_1-\cdots x_6 > x_5$ so
\begin{align*}
J_6 \ \le \ \int_{\overline{P_6}} &\frac{\dd{x_1}\dd{x_2} \dd{x_3}}{x_1 x_2 x_3} \int_{\frac{1}{12}-5\eta<x_6<x_5<x_4<\frac{1}{12}+\eta}\frac{\dd{x_4}\dd{x_5}\dd{x_6}}{x_4x_5^2x_6 (x_5-\eta)},
\end{align*}
where $\overline{P_6}=\overline{P_{6,1}}\cup \overline{P_{6,2}}$ is the (3-dimensional) projection of $P_6$, given explicitly by
\begin{align}
\overline{P_{6,1}} = \{(x_1,x_2,x_3)\in\mathrm D^+(\tfrac{7}{12}+\eta) \; &: \; x_1+x_2 < \tfrac{1}{6}+2\eta \ \text{ and } \ x_3>\tfrac{1}{12}-5\eta \nonumber\\
&\qquad \text{ and } \ x_2+2x_3>\tfrac{1}{4}-3\eta\}, \nonumber\\
\overline{P_{6,2}} = \{(x_1,x_2,x_3)\in\mathrm D^+(\tfrac{7}{12}+\eta) \; &: \; x_1,\,x_2+x_3 < \tfrac{1}{6}+2\eta \ \text{ and } \ x_3>\tfrac{1}{12}-5\eta \nonumber\\
&\qquad  \  \text{ and } \ x_1+x_3,\;x_2+2x_3>\tfrac{1}{4}-3\eta \}. \nonumber
\end{align}

For $\eta = \frac{1}{204}$, we compute $J_6 \  \le \ (J_{6,1}+J_{6,2})J_{6,0}$ where
\begin{align*}
J_{6,0} &=\int_{\frac{1}{12}-5\eta<x_6<x_5<x_4<\frac{1}{12}+\eta}\frac{\dd{x_4}\dd{x_5}\dd{x_6}}{x_4x_5^2x_6 (x_5-\eta)} \ \le \ 2.33838,\\
J_{6,1} &=\int_{\overline{P_{6,1}}} \frac{\dd{x_1}\dd{x_2} \dd{x_3}}{x_1 x_2 x_3} \ \le \ 0.000806853,\\
J_{6,2} &=\int_{\overline{P_{6,2}}} \frac{\dd{x_1}\dd{x_2} \dd{x_3}}{x_1 x_2 x_3} \ \le \ 0.00397946.
\end{align*}
Hence combining with \eqref{eq:I4}, for $s\in[1,3]$ we conclude
\begin{align*}
\frac{F^*(s)}{F(s)} \ \le \  1 + (\tfrac{7}{12}+\eta)\cdot \eta\big(J_4+(J_{6,1}+J_{6,2})J_{6,0}\big) \ \le \ 1.000081.
&\qedhere
\end{align*}
\end{proof}

\section{Factorable remainder, after Iwaniec}

In Theorem \ref{thm:wellfactor}, Iwaniec constructed a well-factorable variant $\widetilde{\lambda}^\pm$ of the weights $\lambda^\pm$ from the (Jurkat--Richert) linear sieve. In this section, we prove Theorem \ref{thm:Iwaniecprogrammable} for the  programmably factorable variant $\widetilde{\lambda}^*$ by adapting Iwaniec's construction, similarly building on the Jurkat--Richert type Proposition \ref{prop:modJurkatRich} that we obtained in the previous section. We shall also prove a technical variation on this result, with a variable level depending on anatomy of the moduli.

To set up the construction, we first adapt \cite[Proposition 12.18]{Opera}.
Denote $P(z,u) = P(z)/P(u) = \prod_{u<p\le z}p$.

\begin{proposition}
Let $\eta>0$, and $D=x^{(\frac{7}{12}+\eta)/(1+\eps+\tau)}$ for $\eps>0$ sufficiently small. Let ${\bf D}_r^*$ be defined by \eqref{eq:calDrstar}. Let $\lambda^{(r)}$ be the standard (upper and lower, for $r$ odd and even, resp.) weights for the linear sieve of level $D^\eps$. Then for $u=D^{\eps^2}$, $\tau = \eps^9$,
\begin{align}\label{eq:propIwaniecstar}
S(\mathcal A, z) & \ \le \ |\mathcal A| V(z)\big\{F^*(s) + O(\eps^5)\big\} \nonumber\\
& + \sum_{0\le r\le \eps^{-2}}\sum_{(D_1,..,D_r)\in\mathbf D_r^*} \frac{(-1)^r}{\gamma(D_1,\ldots,D_r)} \sum_{\substack{p_1\cdots p_r\mid P(z,u)\\D_j< p_j\le D_j^{1+\tau}}}\sum_{\substack{b\mid P(u)\\b\le D^{\eps}}} \lambda^{(r)}(b)\, r_{\mathcal A}(bp_1\cdots p_r).
\end{align}
\end{proposition}
\begin{proof}
First we write
\begin{align}\label{eq:Iwaniecpf}
S(\mathcal A,z) \ \le \ S^*(\mathcal A,z) - \sum_{\textnormal{odd }n\le N}S_n(\mathcal A,z)
\end{align}
for any $N\ge1$, where
\begin{align*}
S^*(\mathcal A,z) := \sum_{0\le r\le \eps^{-2}}(-1)^r \sum_{\substack{u\le p_r<\cdots p_1<z\\p_1\cdots p_r\in \mathcal D^*}} |\mathcal A_{p_1\cdots p_r}|,\quad
S_n(\mathcal A,z) := \sum_{\substack{p_n<\cdots< p_1<z\\p_1\cdots p_{n-1}\in \mathcal D^*\\ p_1\cdots p_n\notin \mathcal D^*}}S(\mathcal A_{p_1\cdots p_n},p_n).
\end{align*}
We apply the inequality \eqref{eq:Iwaniecpf}, not for $\mathcal A=(a_n)$ itself but rather for the subsequence $\tilde{\mathcal A} = (a_n\1_{(n,P(u))=1})$. Here we take $u = D^{\eps^2}$, and then return to $\mathcal A$ by means of the Fundamental Lemma.

Let $z=D^{1/s}$ with $2\le s\le \eps^{-1}$. Since $z>u$, the only change to the above bound \eqref{eq:Iwaniecpf} when passing to $\tilde{\mathcal A}$ is the term $S^*(\tilde{\mathcal A},z)$, provided that $N$ is not too large in terms of $\eps$. Specifically, we require the lower bound for $p_n$ (by induction, the linear sieve conditions imply $p_1\cdots p_m<D^{1-2^{-m}}$)
\begin{align}
p_n \ge (D/p_1\cdots p_{n-1})^{1/3} \ge D^{2^{1-N}/3}
\end{align}
to be larger than $u = D^{\eps^2}$, which certainly holds provided
\begin{align}
N \le \frac{1}{2}\log \frac{1}{\eps}.
\end{align}

Now it remains to evaluate $S^*(\tilde{\mathcal A},z)$,
\begin{align}\label{eq:Sstar}
S^*(\tilde{\mathcal A},z) = \sum_{0\le r\le \eps^{-2}}(-1)^r \sum_{\substack{u\le p_r<\cdots p_1<z\\p_1\cdots p_r\in \mathcal D^*}} |\tilde{\mathcal A}_{p_1\cdots p_r}|.
\end{align}
For each $r$, we break the range in the inner sum into boxes. Namely, we let $D_1,...,D_r$ run over numbers of form
\begin{align}
D^{\eps^2 (1+\tau)^j}, \quad j =0,1,2,\ldots
\end{align}
with $\tau=\eps^9$. We denote by $\mathbf D_r^+=\mathbf D_r^+(D)$ the set of $r$-tuples $(D_1,...,D_r)$ with $D_r\le \cdots \le D_1\le \sqrt{D}$, such that
\begin{align*}
\mathbf D_r^+ & = \begin{cases}
\{(D_1,...,D_r): D_1\cdots D_{m-1}D_m^3 < D \qquad\text{ for all }\text{odd }m\le r\} & \text {if $r$ even},\\
\{(D_1,...,D_r): D_1\cdots D_{m-1}D_m^3 < D^{1/(1+\tau)}\quad \text{for all }\text{odd }m\le r\} &\text {if $r$ odd}.
\end{cases}
\end{align*}
We note, for $\eps>0$ sufficiently small, the cardinalities of the $\mathbf D_r^+$ are bounded by
\begin{align}\label{eq:Drsize}
\sum_{0\le r\le \eps^{-2}}|\mathbf D_r^+| \ \le \ \exp(\eps^{-3}).
\end{align}

Hereafter let $D = x^{(\frac{7}{12}+\eta)/(1+\tau+\eps)}$, and define
\begin{align}\label{eq:calDrstar}
\mathbf D_r^* & = \{(D_1,...,D_r)\in \mathbf D_r^+(D): (D_1,...,D_i)\notin \mathbf P_{i,r} \text{ for }i\le r,i\in\{4,6\}\}.
\end{align}
where $\mathbf P_{4,r}, \mathbf P_{6,r}$ are ($\tau$-enlarged, for even $r$) analogues of the polytopes $\mathcal P_4$, $\mathcal P_6$ in \eqref{eq:defP4P6}, e.g.
\begin{align*}
\mathbf P_{4,r}  & = 
\begin{cases}
\{ (D_1,\ldots, D_4) \; : \;  D_1^{1+\tau} < x^{\frac{1}{6}+2\eta} \ \text{ and } \ (D_2D_4)^{1/(1+\tau)} > x^{\frac{1}{4}-3\eta}\} &\text {if $r$ even},\\
\{ (D_1,\ldots, D_4) \; : \;  D_1 < x^{\frac{1}{6}+2\eta} \ \text{ and } \ D_2D_4>x^{\frac{1}{4}-3\eta}\} &\text{if $r$ odd}.
\end{cases}
\end{align*}

Observe each integer $p_1\cdots p_r$ has a unique vector $(D_1,\ldots,D_r)$ such that $p_i\in (D_i,D_i^{1+\tau}]$ for all $i\le r$, inducing a map $\nu:\mathbb N\to \bigcup_r \mathbf D_r^+$. As a convention $\nu(1)=()$ is the empty vector. By construction, for even $r$ if $p_1\cdots p_4\notin \mathcal P_4$ then $\nu(p_1\cdots p_4)\notin \mathbf P_{4,r}$, and if $(D_1,\ldots, D_4)\notin \mathbf P_{4,r}$ then $\nu^{-1}(D_1,\ldots, D_4)\cap \mathcal P_4=\emptyset$ for odd $r$. Continuing this argument, we deduce
\begin{align}\label{eq:nuDstar}
p_1\cdots p_r\in \mathcal D^* & \ \implies \ \nu(p_1\cdots p_r)\in \mathbf D_r^* &\text{if $r$ even},\\
(D_1,\ldots, D_r)\in \mathbf D_r^* & \ \implies \ \nu^{-1}(D_1,\ldots, D_r) \subset \mathcal D^*  &\text{if $r$ odd}. \nonumber
\end{align}

Without loss, we may restrict $\mathbf D_r^*$ to vectors with nonempty preimage in $\mathcal D^*$. Hence by construction, \eqref{eq:Sstar} becomes\footnote{Indeed, we have reverse engineered the definition of $\mathbf D_r^*$ just so that \eqref{eq:Sstar2} holds.}
\begin{align}\label{eq:Sstar2}
S^*(\tilde{\mathcal A},z) \le \sum_{0\le r\le \eps^{-2}} \sum_{(D_1,..,D_r)\in\mathbf D_r^*} \frac{(-1)^r}{\gamma(D_1,\ldots,D_r)}\sum_{\substack{p_1\cdots p_r\mid P(z)\\D_j< p_j\le D_j^{1+\tau}}} |\tilde{\mathcal A}_{p_1\cdots p_r}|,
\end{align}
where $\gamma(D_1,\ldots,D_r)=k_1!\cdots k_\ell!$ for the corresponding multiplicities $k_i\ge1$ (i.e. we have $r=k_1+\cdots+k_\ell$ and $D_1=\cdots=D_{k_1}<D_{k_1+1}=\cdots=D_{k_2}<\cdots = D_r$.). Note the term $r=0$ corresponds to $|\tilde{\mathcal A}|$ with $p_1=\cdots=p_r=1$.

Now by the Fundamental Lemma \cite[Theorem 6.9]{Opera}, we have upper (and lower) bounds
\begin{align*}
|\tilde{\mathcal A}_{p_1\cdots p_r}| &= S(\mathcal A_{p_1\cdots p_r},u)\\
& \le \; g(p_1\cdots p_r) |\mathcal A|V(u)\{1+O(e^{-1/\eps})\} \ + \ \sum_{b\le D^{\eps}} \lambda^{(r)}(b)\, r_{\mathcal A}(bp_1\cdots p_r),
\end{align*}
(with $\le $ replaced by $\ge$ for the lower bound) where $\lambda^{(r)}$ is the upper (lower) bound $\beta$-sieve of level $D^\eps$ when $r$ is even (odd). For further details on the Fundamental Lemma and $\beta$-sieve, we refer the reader to \cite[\S 6 \& \S 11]{Opera}.

Plugging back into \eqref{eq:Sstar2}, we get
\begin{align}\label{eq:Iwaniecpf2}
S(\mathcal A,z) \ \le \ & S^*(\tilde{\mathcal A},z) \le \sum_{0\le r\le \eps^{-2}} \sum_{(D_1,..,D_r)\in\mathbf D_r^*} \frac{(-1)^r}{\gamma(D_1,\ldots,D_r)}\sum_{\substack{p_1\cdots p_r\mid P(z)\\D_j< p_j\le D_j^{1+\tau}}}\times \\
& \times\bigg\{g(p_1\cdots p_r) |\mathcal A|V(u)\{1+O(e^{-1/\eps})\} \ + \ \sum_{b\le D^{\eps}} \lambda^{(r)}(b)\, r_{\mathcal A}(bp_1\cdots p_r)\bigg\}. \nonumber
\end{align}

We now compare the main term above to that of the modified linear sieve, as in \eqref{eq:Vstardef}--\eqref{eq:VFstar} from the proof of Proposition \ref{prop:modJurkatRich}, namely,
\begin{align*}
V^*(D,z) := \sum_{d\mid P(z)/P(u)} \lambda^*(d) g(d) \ = \ V(z)\big\{F^*(s) + o(1)\big\}.
\end{align*}
The difference between these main terms is accounted for by those $d$ with two close prime factors, within a ratio $D^{\tau}$, and those $d$ near the boundary. The former contribution is
\begin{align*}
\sum_{\substack{d\mid P(z)/P(u)\\ p\le p'<pD^\tau\\pp'\mid d}} g(d) \le \sum_{\substack{u\le p<z\\ p\le p'<pD^\tau}} g(pp')\cdot \prod_{u\le p<z}(1+g(p)),
\end{align*}
and the latter contribution is
\begin{align}\label{eq:boundary}
\sum_r \sum_{\substack{u<p_r<\cdots p_1<z\\D^{1/(1+\tau)}<p_1\cdots p_m^3<D}} g(p_1\cdots p_r),
\end{align}
where $m$ is the first index ($m\le r$) for which this occurs. Both contributions may be shown to be $O(\eps^5)$, see \cite[pp. 254-255]{Opera}. Hence the Proposition follows.
\end{proof}
\begin{remark}
We make a minor technical point. Namely, at an admissible cost $O(\eps^5)$ we may assume Proposition \ref{eq:propIwaniecstar} holds, where $\mathbf D_r^*$ is further restricted to vectors $(D_1,\ldots, D_r)$ satisfying
\begin{align}\label{eq:rmknu}
\nu^{-1}(D_1^{\frac{1}{1+\tau}},\ldots, D_r^{\frac{1}{1+\tau}}) \subset \mathcal D^*,
\end{align}
regardless of parity of $r$. To show this, note by definitions of $\mathbf P_{4,r}, \mathbf P_{6,r}$ the integers $p_1\cdots p_r$ with $p_j\in [D_j^{1/(1+\tau)},D_j]$, $j\le r$, that lie outside $\mathcal D^*$ must satisfy
\begin{align*}
x^{\frac{1}{6}+2\eta}/p_1 \quad \textnormal{or} \quad
x^{\frac{1}{12}-5\eta}/p_6 \ \in \ [x^{-2\tau},x^{2\tau}].
\end{align*}
Then for $B_1 = x^{\frac{1}{6}+2\eta+2\tau}$, $B_6=x^{\frac{1}{12}-5\eta+2\tau}$, we have
\begin{align*}
\sum_{B_i\ge p_i \ge \max(B/x^{4\tau},u)}g(p_i) \ \ll \
\log \frac{\log B_i}{\log \max(B_i/x^{4\tau},u)} \ll
\frac{\log x^{4\tau}}{\log u} \ll \frac{\tau}{\eps^2},
\end{align*}
and so the contribution of such integers to the main term of \eqref{eq:Iwaniecpf2} is
\begin{align*}
\ll \sum_r \sum_{\substack{u<p_r<\cdots p_1<z\\B_i\ge p_i >\max(B_i/x^{4\tau},u),\;i\in\{1,6\}}} g(p_1\cdots p_r)
\ll \frac{\tau}{\eps^2}\prod_{u<p<z}\big(1+g(p)\big)
\ll \frac{\tau}{\eps^2}\frac{\log z}{\log u} \ll \eps^5.
\end{align*}
Hence \eqref{eq:rmknu} follows.
\end{remark}

We now proceed to Theorem \ref{thm:Iwaniecprogrammable}. For each vector $(D_1,..,D_r)\in \mathbf D_r$ we define the weight $\lambda_{(D_1,..,D_r)}$ supported on $d$ in $\nu^{-1}(D_1,..,D_r)$, namely,
\begin{align}\label{eq:lamD1D}
\lambda_{(D_1,..,D_r)}(d) \ := \ \1^{d=p_1\cdots p_r}_{D_j< p_j\le\; D_j^{1+\tau}\;\forall j}.
\end{align}
Next, we may decompose an integer $d$ into its $u$-smooth and rough components, $d = b(p_1\cdots p_r)$. Recall $\mathcal D^{(r)}=\mathcal D^\pm$, $\lambda^{(r)}=\lambda^\pm$ (depending on the parity of $r$), and for $b\mid P(u)$ we have $\lambda^{(r)}(b)=\mu(b)$ if $b\in \mathcal D^{(r)}$, and $\lambda^{(r)}(b)=0$ else. Thus we may define the convolution $\lambda_{(D_1,..,D_r)}^{(r)}:=\lambda_{(D_1,..,D_r)}\ast \lambda^{(r)}$, i.e.
\begin{align}\label{eq:lamD1Dr}
\lambda_{(D_1,..,D_r)}^{(r)}(d) \ = \
\begin{cases}
\mu(b) & \text{if }d=bp_1\cdots p_r \ \text{ for } \ b\in \mathcal D^{(r)}(D^\eps), \ b\mid P(D^{\eps^2}),\\
& \qquad\qquad\qquad\quad \text{ and }\, D_j< p_j\le\; D_j^{1+\tau}\;\forall j\le r.\\
0 & \text{else.}
\end{cases}
\end{align}
Hence the remainder in \eqref{eq:propIwaniecstar} equals
\begin{align*}
\sum_{0\le r\le \eps^{-2}}\sum_{(D_1,..,D_r)\in\mathbf D_r^*} \frac{(-1)^r}{\gamma(D_1,\ldots,D_r)} \sum_{d\mid P(z)} \lambda_{(D_1,..,D_r)}^{(r)}(d)\, r_{\mathcal A}(d) 
\ = \ \sum_{d\mid P(z)}\widetilde{\lambda}^*(d)\, r_{\mathcal A}(d),
\end{align*}
for the weights
\begin{align}\label{def:tildelambdastar}
\widetilde{\lambda}^* = \sum_{0\le r\le \eps^{-2}}\sum_{(D_1,..,D_r)\in\mathbf D_r^*} \frac{(-1)^r}{\gamma(D_1,\ldots,D_r)} \lambda_{(D_1,..,D_r)}^{(r)}.
\end{align}
Recalling the cardinality of $\mathbf D_r^*\subset \mathbf D_r^+$ from \eqref{eq:Drsize}, it suffices to show the weights $\lambda_{(D_1,..,D_r)}^{(r)}$ are programmably factorable for each vector in $\mathbf D_r^*$. To this we have the following.
\begin{lemma}\label{lem:factorDstarr}
For an integer $d$ denote the vector $\nu(d)=(D_1,\ldots, D_r)$ from \eqref{eq:nuDstar}. If $d$ has a factorization as in \eqref{eq:JMfactorable} at level $D$, then the corresponding weights $\lambda_{\nu(d)}$ and $\lambda_{\nu(d)}^{(r)}$, as in \eqref{eq:lamD1D} and \eqref{eq:lamD1Dr}, resp., are programmably factorable sequences of levels $D^{1+\tau}$ and $D^{1+\tau+\eps}=x^{\frac{7}{12}+\eta}$, resp. (relative to $x$, $\eps/50$).
\end{lemma}
\begin{proof}
By assumption for each $N\in [1,x^{1/3}]$, there is a factorization $d=d_1d_2d_3$ satisfying the system \eqref{eq:JMfactorable}. For $j=1,2,3$, write $d_j = \prod_{i\in I_j}p_i$ for the induced partition of indices $\{1,...,r\}=I_1\cup I_2\cup I_3$.
Thus letting $Q_j = \prod_{i\in I_j} D_i$, the factorization $D_1\cdots D_r = Q_1Q_2Q_3$ satisfies \eqref{eq:defprogram}, since $D_i<p_i$.

Further, writing the corresponding subvectors $(D_i)_{i\in I_j}$ for $j=1,2,3$, the weights $\lambda_{(D_i)_{i\in I_j}}$ are 1-bounded, supported on $[1,Q_j^{1+\tau}]$, and give the desired triple convolution,
\begin{align*}
\lambda_{(D_1,\ldots,D_r)} = \lambda_{(D_i)_{i\in I_1}}\ast\lambda_{(D_i)_{i\in I_2}}\ast\lambda_{(D_i)_{i\in I_3}}.
\end{align*}
Hence $\lambda_{(D_1,\ldots,D_r)}$ is programmably factorable of level $D^{1+\tau}$ as claimed. Similarly $\lambda_{(D_1,\ldots,D_r)}^{(r)} = \lambda_{(D_i)_{i\in I_1}}\ast\lambda_{(D_i)_{i\in I_2}}\ast\lambda_{(D_i)_{i\in I_3}}^{(r)}$ is programmably factorable of level $D^{1+\tau+\eps}$.
\end{proof}

Now for each vector $(D_1,\ldots,D_r)\in\mathbf D^*_r$, by \eqref{eq:rmknu} there exists $d=p_1\cdots p_r\in\mathcal D^*$ for some primes $p_i\in(D_i^{1/(1+\tau)},D_i]$. Then for all $N\in [1,x^{1/3}]$ Proposition \ref{prop:factorD} gives a factorization of $d$ as in \eqref{eq:JMfactorable}, and so by Lemma \ref{lem:factorDstarr} $\lambda_{(D_1,\ldots,D_r)}^{(r)}$ is programmably factorable of level $x^{\frac{7}{12}+\eta}$.

This completes the proof of Theorem \ref{thm:Iwaniecprogrammable}.

\subsection{Variable level of distribution for the linear sieve weights}

We now return to Iwaniec's well-factorable weights $\widetilde{\lambda}^\pm$ for the (upper and lower) linear sieve, given explicitly from \eqref{eq:lamD1Dr} as the weighted sum,
\begin{align}\label{eq:lambdapm}
\widetilde{\lambda}^\pm = \sum_{0\le r\le \eps^{-2}} \sum_{(D_1,..,D_r)\in\mathbf D_r^\pm} \frac{(-1)^r}{\gamma(D_1,\ldots,D_r)} \;\lambda^{(r)}_{(D_1,..,D_r)}.
\end{align}
We introduce the analogous set of well-factorable vectors $\mathbf D_r^{\textnormal{well}}=\mathbf D_r^{\textnormal{well}}(D)$,
\begin{align}
\mathbf D_r^{\textnormal{well}} & =
\{(D_1,...,D_r): D_1\cdots D_{m-1}D_m^2 < D \quad\text{ for all }m\le r\}.
\end{align}
Note $\mathbf D_r^\pm\subset \mathbf D_r^{\textnormal{well}}$, having dropped parity conditions on the indices $m\le r$.

We have the following technical variation on Theorem \ref{thm:Iwaniecprogrammable} for the original linear sieve.

\begin{proposition}\label{prop:lvlalpha1}
Let $(D_1,...,D_r)\in\mathbf D_r^{\textnormal{well}}(D)$ and write $D=x^\theta$, $D_i=x^{t_i}$ for $i\le r$. If $\theta \le \theta(t_1)-\eps$ as in \eqref{eq:thetat}, then
\begin{align}\label{eq:lvlalpha1}
\sum_{\substack{b=p_1\cdots p_r\\ D_i< p_i \le D_i^{1+\tau}}}
\sum_{\substack{d=bc \leq x^{\theta}\\ c\mid P(p_r)\\(d,a)=1}} \widetilde{\lambda}^\pm(d) \,\Big(\pi(x;d,a)-\frac{\pi(x)}{\phi(d)}\Big) \ \ll_{a,A,\eps} \ 
 \frac{x}{(\log x)^A}.
\end{align}
Moreover if $t_1 \le \,\frac{1}{5}$ and $r\ge3$, then \eqref{eq:lvlalpha1} holds provided that $\theta \, \le \, \theta(t_1,t_2,t_3)-\eps$ as in \eqref{eq:thetat123}. 

If $t_1 \le \,\frac{1}{5}$ and $r\le 2$, then provided $\theta \le \frac{3-u}{5}-\eps$,
\begin{align*}
\sum_{\substack{b=p_1\cdots p_r\\ D_i< p_i \le D_i^{1+\tau}}}
\sum_{\substack{d=bc \le x^\theta\\c\mid P(x^u)\\(d,a)=1}} &\widetilde{\lambda}^\pm(d) \,\Big(\pi(x;d,a)-\frac{\pi(x)}{\phi(d)}\Big) \ \ll_{a,A,\eps}
\frac{x}{(\log x)^A}.
\end{align*}
In particular for $r=0$ (i.e. the empty vector), $\theta \le \frac{3-u}{5}-\eps$, this simplifies as
\begin{align*}
\sum_{\substack{d\le x^\theta\\d\mid P(x^u)\\(d,a)=1}} &\widetilde{\lambda}^\pm(d) \,\Big(\pi(x;d,a)-\frac{\pi(x)}{\phi(d)}\Big) \ \ll_{a,A,\eps}
\frac{x}{(\log x)^A}.
\end{align*}
\end{proposition}
\begin{proof}
Given $(D_1,...,D_r)$, take an integer $b=p_1\cdots p_r$ with $D_i<p_i\le D_i^{1+\tau}$. Then for all multiples $d$ of $b$ with $d/b\mid P(p_r)$, the weight $\lambda^{(s)}_{(D'_1,..,D'_s)}(d)$ vanishes unless the vector $(D'_1,..,D'_s)$ extends $(D_1,...,D_r)$. That is, $D'_i=D_i$ for all $i\le r$. Conversely, given such a vector $(D'_1,..,D'_s)$ we have $\lambda^{(s)}_{(D'_1,..,D'_s)}(d')=0$ unless the first $s$ primes of $d'$ are $p_1\cdots p_s$ with $D_i<p_i\le D_i^{1+\tau}$, $i\le r$. So by definition of $\widetilde{\lambda}^\pm$ as in \eqref{eq:lambdapm}, we have
\begin{align}\label{eq:lambp1}
\sum_{\substack{b=p_1\cdots p_r\\ D_i< p_i \le D_i^{1+\tau}}}\sum_{\substack{d=bc\le x^\theta\\c\mid P(p_r)\\(d,a)=1}} &\widetilde{\lambda}^\pm(d) = \sum_{r\le s\le \eps^{-2}} \sum_{\substack{(D'_1,\ldots,D'_s)\in\mathbf D_s^\pm\\D'_i=D_i, \, i\le r}} \frac{(-1)^s}{\gamma(D_1',\ldots,D'_s)} 
\sum_{\substack{d\le x^\theta\\(d,a)=1}}\lambda^{(s)}_{(D_1',..,D'_s)}(d).
\end{align}
Here we have extended (by zero) the inner sum to all $d\le x^\theta$, $(d,a)=1$.

Next, take such a vector $(D'_1,...,D'_s)\in\mathbf D_s^\pm(x^\theta)$ with $D'_i=D_i$ for $i\le r$. Each integer $d=p_1\cdots p_s$ with $D_i'< p_i\le (D_i')^{1+\tau}$ lies in $d\in \mathcal D^{\textnormal{well}}(x^{\theta+\tau})$. In particular $p_1\le D_1^{1+\tau}\le x^{t_1+\tau}$ so by Corollary \ref{cor:piecewiseeta}, $d$ has a factorization as in \eqref{eq:JMfactorable} at level $x^{\theta(t_1+\tau)}$. Since $\theta(t)$ is continuous (in fact, piecewise linear), for $\tau>0$ sufficiently small $\theta(t_1+\tau)\ge\, \theta(t_1)-\eps \ge\, \theta$. Thus by Lemma \ref{lem:factorDstarr} the weights $\lambda^{(s)}_{(D_1,..,D_s)}$ are programmably factorable sequences of level $x^{\theta}$. Hence for each such vector, by Theorem \ref{thm:Maynardprogram} we have
\begin{align}\label{eq:lambp2}
\sum_{\substack{d\le x^\theta\\(d,a)=1}}\lambda^{(s)}_{(D_1',..,D'_s)}(d)
\ \ll_{a,A,\eps} \ \frac{x}{(\log x)^A}.
\end{align}
Plugging \eqref{eq:lambp2} back into \eqref{eq:lambp1} gives the bound \eqref{eq:lvlalpha1}, as claimed.

Moreover, if $t_1\le \frac{1}{5}$ and $r\ge3$ then proceeding as in the above paragraph, by Corollary \ref{cor:piecewiseeta} $d$ will factor as in \eqref{eq:JMfactorable} to level $x^{\theta(t_1+\tau,t_2+\tau,\tau_3+\tau)}$. Again $\theta(t,u,v)$ is continuous, so for $\tau>0$ sufficiently small $\theta(t_1+\tau,t_2+\tau,\tau_3+\tau)\ge\, \theta(t_1,t_2,t_3)-\eps \ge\, \theta$. Hence \eqref{eq:lvlalpha1} also follows in this case.

Similarly if $t_1\le \frac{1}{5}$ and $r\le2$, proceeding as above with $d=p_1\cdots p_s$, the assumption $d/b\mid P(x^u)$ implies $s\le 2$ or $p_3\le D_3^{1+\tau}\le x^{u+\tau}$. Thus by Corollary \ref{cor:piecewiseeta} $d$ will factor as in \eqref{eq:JMfactorable} to level $x^{(3-u-\tau)/5} \ge x^\theta$. Hence \eqref{eq:lvlalpha1} follows in this case as well.
\end{proof}

\subsection{Equidistribution for products of primes}

We use an extension of Theorem \ref{thm:Maynardprogram} to products of $k$ primes. This is the analogue in the programmably factorable setting of Lemma 7 \cite{WuI} extending Theorem \ref{thm:BFIwell} of Bombieri--Friedlander--Iwaniec.

\begin{proposition}\label{prop:JMprogramprimes}
Let $\eps>0$ and $\lambda$ be a programmably factorable sequence of level $D \le x^{\frac{3}{5}-\eps}$ (relative to $x$, $\eps/50$). Take real numbers $\eps_1,\ldots, \eps_k\ge \eps$ such that
$\sum_{i\le k}\eps_i=1$. Then for any fixed integer $a\in\Z$, $A,B>0$, letting $\Delta = 1 + (\log x)^{-B}$ we have
\begin{align}\label{eq:programprimes}
\sum_{\substack{d\le D\\(d,a)=1}}\lambda(d) \bigg(
\sum_{\substack{p_1\cdots p_k \equiv a\,({\rm mod}\,d)\\ x^{\eps_i}/\Delta<p_i\le x^{\eps_i} \,\forall i\le k}}1
\ - \ \frac{1}{\phi(d)}\sum_{\substack{(p_1\cdots p_k,d)=1\\ x^{\eps_i}/\Delta<p_i\le x^{\eps_i} \,\forall i\le k}}1\bigg) \ \ll_{a,\eps,A,B} \ \frac{x}{(\log x)^A}.
\end{align}
\end{proposition}
\begin{proof}
This follows by the same proof method as in Theorem \ref{thm:Maynardprogram} (i.e. Maynard's \cite[Theorem 1.1]{JM2}). Indeed, Maynard just uses the Heath--Brown identity to decompose the indicator function of primes into Type I/II sums. A similar decomposition holds for products of $k$ primes, after which we may apply the same Type I/II estimates in Propositions 5.1 and 5.2 of \cite{JM2}.
\end{proof}

In addition, by replacing Theorem \ref{thm:Maynardprogram} with Proposition \ref{prop:JMprogramprimes} in the proof, we obtain analogues of Proposition \ref{prop:lvlalpha1}  for the linear sieve weights $\lambda=\widetilde{\lambda}^\pm$ in the case of products of $k$ primes.

\begin{corollary}\label{cor:lvlalphak}
Let $(D_1,...,D_r)\in\mathbf D_r^{\textnormal{well}}(D)$ and write $D=x^\theta$, $D_i=x^{t_i}$ for $i\le r$. Let $\eps>0$ and real numbers $\eps_1,\ldots, \eps_k\ge \eps$ such that
$\sum_{i\le k}\eps_i=1$. Fix an integer $a\in\Z$, $A,B>0$, and let $\Delta = 1 + (\log x)^{-B}$.  If $\theta \le \theta(t_1)-\eps$ as in \eqref{eq:thetat},
\begin{align}\label{eq:lvlalphak}
\sum_{\substack{b=p_1'\cdots p_r'\\ D_i< p_i' \le D_i^{1+\tau}}}\sum_{\substack{d=bc\le x^\theta\\c\mid P(p_r')\\(d,a)=1}} \widetilde{\lambda}^\pm(d) \,\bigg(
\sum_{\substack{p_1\cdots p_k \equiv a\,({\rm mod}\,d)\\ x^{\eps_i}/\Delta<p_i\le x^{\eps_i} \,\forall i\le k}}1
\ &- \ \frac{1}{\phi(d)}\sum_{\substack{(p_1\cdots p_k,d)=1\\ x^{\eps_i}/\Delta<p_i\le x^{\eps_i} \,\forall i\le k}}1\bigg) \\
& \ll_{a,\eps,A,B} \ \frac{x}{(\log x)^A}. \nonumber
\end{align}
Moreover if $t_1 \le \,\frac{1}{5}$, $r\ge3$, then \eqref{eq:lvlalphak} holds provided that $\theta \, \le \, \theta(t_1,t_2,t_3)-\eps$ as in \eqref{eq:thetat123}.

In addition, if $r\le 2$, $u\le t_r$, $t_1 \le \,\frac{1}{5}$, and $\theta \le \frac{3-u}{5}-\eps$, then
\begin{align}\label{eq:lvlalphaksmooth}
\sum_{\substack{b=p_1'\cdots p_r'\\ D_i< p_i' \le D_i^{1+\tau}}}\sum_{\substack{d=bc\le x^\theta\\c\mid P(x^u)\\(d,a)=1}} \widetilde{\lambda}^\pm(d) \,\bigg(
\sum_{\substack{p_1\cdots p_k \equiv a\,({\rm mod}\,d)\\ x^{\eps_i}/\Delta<p_i\le  x^{\eps_i} \,\forall i\le k}}1
\ & - \ \frac{1}{\phi(d)}\sum_{\substack{(p_1\cdots p_k,d)=1\\ x^{\eps_i}/\Delta<p_i\le  x^{\eps_i} \,\forall i\le k}}1\bigg) \\
& \ll_{a,\eps,A,B} \ \frac{x}{(\log x)^A}.\nonumber
\end{align}
\end{corollary}

\section{Upper bound for twin primes}

Now we shall apply the modified sieve to the set of twin primes
\begin{align*}
\mathcal A := \{p+2 : p\le x\}.
\end{align*}
In this case the sieve notation specializes as $\mathcal P = \{p>2\}$ and $g(d)=1/\phi(d)$ for odd $d$, so that $V(z) = \prod_{2<p<z}(1-1/\phi(p))$. Recall $V(z)\sim \mathfrak{S}_2/e^\gamma \log z$ by Mertens theorem, for the Hardy--Littlewood constant $\mathfrak{S}_2=2\prod_{p>2}\frac{1-2/p}{(1-1/p)^2}$ appearing in $\Pi(x)=\frak{S}_2\, x/(\log x)^2$.

We begin in the spirit of Fouvry--Grupp \cite{FG}, and apply a weighted sieve inequality. To each non-switched term, we apply the Buchstab identity in order to lower the sieve threshhold down to $z=x^\epsilon$ for some tiny $\epsilon>0$. By Proposition \ref{prop:lvlalpha1}, such smooth sums will satisfy level of distribution $\frac{3-\epsilon}{5}$. Combined with variations on a theme, which identify programmably-factorable weights in certain cases, the consequent increase in level will be sufficient to obtain the bound in Theorem \ref{thm:twinbound}. For a final bit of savings, we use refinements obtained by Wu's iteration method \cite{WuII}.

Before moving on, we note that sieve methods and the switching principle have also yielded progress on the Goldbach problem. Indeed, the upper bound in Wu \cite[Theorem 3]{WuII} for twin primes is obtained by using the same formulae as in \cite[Theorem 1]{WuII} for the Goldbach problem, except for altering the level from $\frac{1}{2}$ to $\frac{4}{7}$ (this amounts to replacing factors of $4$ with $7/2$ in a few instances). Importantly, the quantitative upper bounds for twin primes are much stronger than those of the Goldbach problem. This is because the latter relies on level $\frac{1}{2}$ from Bombieri--Vinogradov for a growing residue $a=N$, while the former may appeal to level of distribution $\frac{4}{7}$ from Bombieri--Friedlander--Iwaniec for the fixed residue $a=2$ (and now the subsequent improvements of Maynard). As such our methods have nothing new to say for the Goldbach problem.

\subsection{Lemmas}

We begin with a standard lemma for $x^{1/u}$-rough numbers in terms of the Buchstab function $\omega(u) = \big(f(u)+F(u)\big)/(2e^\gamma)$ for linear sieve functions $f$, $F$ as in Theorem \ref{thm:wellfactor}. Alternatively, $\omega$ is directly defined via
\begin{align*}
\omega(u) & = \frac{1}{u} & \text{for} \quad 1\le u\le2,\\
\big(u\,\omega(u)\big)' &= \omega(u-1) & \text{for} \ \quad 2\le u.\quad
\end{align*}
\begin{lemma}\label{lem:Buchstab}
Let $x\ge2$ and $y=x^{1/u}$. Then we have
\begin{align*}
\sum_{\substack{n\le x\\p\mid n\Rightarrow p\ge y}}1 = \omega(u)\,\frac{x}{\log y} + O\Big(\frac{x}{(\log y)^2}\Big).
\end{align*}
\end{lemma}
\begin{proof}
This is \cite[Lemma 12]{WuI}.
\end{proof}

The argument of Wu makes essential use of weighted sieve inequalities, as in Lemmas 4.1 and 4.2 \cite{WuII}. We shall employ the latter inequality in the special case $d=1$, $\sigma=1$.
\begin{lemma}\label{lem:WuII42}
For $3/10\ge \rho\ge \tau_3 > \tau_2 > \tau_1\ge \rho'\ge 1/20$, we have
\begin{align*}
5S(\mathcal A, x^{\rho}) \ \lesssim \ \sum_{1\le n\le 21}\Gamma_n,
\end{align*}
where

\begin{equation*}
\begin{aligned}[c]
\Gamma_1 := & \ 4S(\mathcal A, x^{\rho'}) \ + \ S(\mathcal A, x^{\tau_1}),\\
\Gamma_2 := & -\sum_{x^{\rho'} \le p < x^{\rho}} S(\mathcal A_{p}, x^{\rho'}),\\
\Gamma_3 := & -\sum_{x^{\rho'} \le p < x^{\tau_2}} S(\mathcal A_{p}, x^{\rho'}),\\
\Gamma_4 := & -\sum_{x^{\rho'} \le p < x^{\tau_3}} S(\mathcal A_{p}, x^{\rho'}),\\
\Gamma_5 := & \underset{x^{\rho'} \le p_1 < p_2 < x^{\tau_2}}{\sum\sum} S(\mathcal A_{p_1p_2},  x^{\rho'}),\\
\Gamma_6 := & \underset{\substack{x^{\rho'} \le p_1 < x^{\tau_1} \\ x^{\tau_2}\le p_2 < x^{\tau_3}}}{\sum\sum} S(\mathcal A_{p_1p_2}, x^{\rho'}),\\
\Gamma_7 := &  \underset{\substack{x^{\rho'} \le p_1 < p_2< x^{\tau_1}}}{\sum\sum} S(\mathcal A_{p_1p_2}, p_1),\\
\Gamma_8 := & \underset{\substack{x^{\rho'} \le p_1 < x^{\tau_1} \le p_2 < x^{\tau_2}}}{\sum\sum} S(\mathcal A_{p_1p_2},  p_1),
\end{aligned}
\quad
\begin{aligned}[c]
\Gamma_9 := & \underset{\substack{x^{\tau_1} \le p_1 <p_2 <p_3 < x^{\tau_3}}}{\sum\sum\sum} S(\mathcal A_{p_1p_2p_3}, p_2),\\
\Gamma_{10} := & \underset{\substack{x^{\tau_1} \le p_1 <p_2 <x^{\tau_2}\le p_3 < x^{\rho}}}{\sum\sum\sum} S(\mathcal A_{p_1p_2p_3}, p_2),\\
\Gamma_{11} := & \underset{\substack{x^{\tau_1} \le p_1 <x^{\tau_2}\le p_2 < p_3 < x^{\tau_3}}}{\sum\sum\sum} S(\mathcal A_{p_1p_2p_3},  p_2),\\
\Gamma_{12} := & \underset{\substack{x^{\rho'} \le p_1 <p_2 <x^{\tau_1}\\ x^{\tau_3} \le p_3 < x^{\rho}}}{\sum\sum\sum} S(\mathcal A_{p_1p_2p_3}, p_2),\\
\Gamma_{13} := & \underset{\substack{x^{\rho'} \le p_1 <x^{\tau_1}\le p_2 < x^{\tau_2}\le p_3 < x^{\rho}}}{\sum\sum\sum} S(\mathcal A_{p_1p_2p_3}, p_2),\\
\Gamma_{14} := & \underset{\substack{x^{\rho'} \le p_1 <x^{\tau_1}\\ x^{\tau_2}\le p_2< p_3 < x^{\rho}}}{\sum\sum\sum} S(\mathcal A_{p_1p_2p_3}, p_2),\\
\Gamma_{15} := & \underset{\substack{x^{\tau_1} \le p_1 <x^{\tau_2}\le p_2 < x^{\tau_3}\le p_3 < x^{\rho}}}{\sum\sum\sum} S(\mathcal A_{p_1p_2p_3}, p_2),
\end{aligned}
\end{equation*}

\begin{align*}
\Gamma_{16} := & \underset{\substack{x^{\tau_2} \le p_1 <p_2 < p_3 < p_4< x^{\tau_3}}}{\sum\sum\sum\sum} S(\mathcal A_{p_1p_2p_3p_4}, p_3),\\
\Gamma_{17} := & \underset{\substack{x^{\tau_2} \le p_1 <p_2 < p_3 < x^{\tau_3} \le p_4< x^{\rho}}}{\sum\sum\sum\sum} S(\mathcal A_{p_1p_2p_3p_4}, p_3),\\
\Gamma_{18} := & \underset{\substack{x^{\tau_2} \le p_1 <p_2 < x^{\tau_3} \le p_3 < p_4< x^{\rho}}}{\sum\sum\sum\sum} S(\mathcal A_{p_1p_2p_3p_4}, p_3),\\
\Gamma_{19} := & \underset{\substack{x^{\tau_1} \le p_1 x^{\tau_2} \\ x^{\tau_3} \le p_2 < p_3 < p_4< x^{\rho}}}{\sum\sum\sum\sum} S(\mathcal A_{p_1p_2p_3p_4},  p_3),\\
\Gamma_{20} := & \underset{\substack{x^{\tau_2} \le p_1 < x^{\tau_3} \le p_2 < p_3 < p_4< p_5 < x^{\rho}}}{\sum\sum\sum\sum\sum} S(\mathcal A_{p_1p_2p_3p_4p_5}, p_4),\\
\Gamma_{21} := & \underset{\substack{x^{\tau_3} \le p_1 <p_2 < p_3 < p_4< p_5<p_6 <x^{\rho}}}{\sum\sum\sum\sum\sum\sum} S(\mathcal A_{p_1p_2p_3p_4p_5p_6}, p_5).
\end{align*}
\end{lemma}
\begin{proof}
This is \cite[Lemma 4.2]{WuII} from Wu with $d=1$, $\sigma=1$. Here we simplify Wu's notation slightly, using
$(\underline{d}^{1/s},\underline{d}^{1/\kappa_3} , \underline{d}^{1/\kappa_2}, \underline{d}^{1/\kappa_1}, \underline{d}^{1/s'}) = (x^\rho, x^{\tau_3},x^{\tau_2},x^{\tau_1},x^{\rho'})$. The basic proof idea is to iterate the Buchstab identity and to strategically neglect some terms by positivity.
\end{proof}

\subsection{Computations}

Given $0.1\le \rho' \le \tau_1 < 0.2 \le \tau_2 < \tau_3 \le \rho \le 0.3.$, we define integrals $I_{n}=I_{n}(\rho,\rho',\tau_1,\tau_2,\tau_3)$ by
\begin{align}\label{eq:defIn}
I_n &= \int_{\D_{n}}\omega\Big(\frac{1-t-u-v}{u}\Big)\frac{\dd{t}\dd{u}\dd{v}}{tu^2v} \quad\qquad\qquad(9\le n\le 15),\nonumber\\
I_n &= \int_{\D_{n}}\omega\Big(\frac{1-t-u-v-w}{v}\Big)\frac{\dd{t}\dd{u}\dd{v}\dd{w}}{tuv^2w} \quad\qquad(16\le n\le19), \\
I_{20} &= \int_{\D_{20}}\omega\Big(\frac{1-t-u-v-w-x}{w}\Big)\frac{\dd{t}\dd{u}\dd{v}\dd{w}\dd{x}}{tuvw^2x},\nonumber\\
I_{21} &= \int_{\D_{21}}\omega\Big(\frac{1-t-u-v-w-x-y}{x}\Big)\frac{\dd{t}\dd{u}\dd{v}\dd{w}\dd{x}\dd{y}}{tuvwx^2y},\nonumber
\end{align}
where $\omega$ is the Buchstab function, and where the domains $\D_{n}$ are
\begin{align*}
\D_{9} &= \{(t,u,v) : \tau_1 < t<u<v < \tau_3\}, \\
\D_{10} &= \{(t,u,v) : \tau_1 < t<u< \tau_2 < v < \rho\}, \\
\D_{11} &= \{(t,u,v) : \tau_1 < t<\tau_2<u < v < \tau_3\}, \\
\D_{12} &= \{(t,u,v) : \rho' < t<u<\tau_1, \ \tau_3 < v < \rho\}, \\
\D_{13} &= \{(t,u,v) : \rho' < t<\tau_1 < u <\tau_2< v < \rho\}, \\
\D_{14} &= \{(t,u,v) : \rho' < t<\tau_1, \ \tau_2 < u < v < \rho\}, \\
\D_{15} &= \{(t,u,v) : \tau_1 < t<\tau_2 < u < \tau_3 <  v < \rho\}, \\
\D_{16} &= \{(t,u,v,w) : \tau_2 < t<u<v<w<\tau_3\}, \\
\D_{17} &= \{(t,u,v,w) : \tau_2 < t<u<v<\tau_3<w<\rho\}, \\
\D_{18} &= \{(t,u,v,w) : \tau_2 < t<u<\tau_3<v<w<\rho\}, \\
\D_{19} &= \{(t,u,v,w) : \tau_1 < t<\tau_2, \ \tau_3 <u<v<w<\rho\}, \\
\D_{20} &= \{(t,u,v,w,x) : \tau_2 < t<\tau_3 < u<v<w<x<\rho\}, \\
\D_{21} &= \{(t,u,v,w,x,y) : \tau_3 < t< u<v<w<x<y<\rho\}.
\end{align*}

Recall the definitions \eqref{eq:thetat} and \eqref{eq:thetat123},
\begin{align*}
\theta(t) = \begin{cases}
\frac{2-t}{3} & \text{if} \ \ t > \frac{1}{5},\\
\frac{1+t}{2} & \text{if} \ \ t \le \, \frac{1}{5}.
\end{cases}
\end{align*}
We let $\theta_{\epsilon} = \frac{3-\epsilon}{5}$ and
\begin{align*}
\theta(t,u,v) := \max\Big\{ \, &\frac{3-v}{5},\,\theta(t),\,\theta(u),\,\theta(t+u+v),\\
& \ \ \theta(t+u),\,\theta(t+v), \,\theta(u+v)\Big\}.
\end{align*}

We also define
\begin{align}\label{eq:defG14}
G_1 &= 4G(\rho') + G(\tau_1), & G_3 &= G_0 + \overline{G}(\tau_2),\\
G_2 &= G_0 + \overline{G}(\rho), & G_4 &= G_0 + \overline{G}(\tau_3),\nonumber
\end{align}
where for $c\le 1/5$,
\begin{align}\label{eq:Gc}
G(c) &= \frac{1}{\epsilon}\,F\big(\theta_\epsilon/\epsilon\big) - \frac{1}{\epsilon}\int_{\epsilon}^{c}\frac{\dd{t}}{t}f\big((\theta_\epsilon - t)/\epsilon\big) + \frac{1}{\epsilon}\int_{\epsilon}^{c}\int_{\epsilon}^{t}\frac{\dd{t}\dd{u}}{t u} F\big((\theta_\epsilon - t-u)/\epsilon\big)\\
&\
 -\int_{\epsilon}^{c}\int_{\epsilon}^{t}\int_{\epsilon}^u \frac{\dd{t}\dd{u}\dd{v}}{t u v^2} f\big((\theta(t,u,v) - t-u-v)/v\big),\nonumber
\end{align}
and for $c>1/5$,
\begin{align}\label{eq:barGc}
\overline{G}(c) &= - \frac{1}{\epsilon}\int_{1/5}^{c}\frac{\dd{t}}{t}f\big((\theta(t) - t)/\epsilon\big) + \int_{1/5}^{c}\int_{\epsilon}^{\rho'}\frac{\dd{t}\dd{u}}{tu^2}F\big((\theta(t) - t-u)/u\big)
\end{align}
as well as
\begin{align}
G_0 &= - \frac{1}{\epsilon}\int_{\rho'}^{1/5}\frac{\dd{t}}{t}f\big((\theta_{\epsilon} - t)/\epsilon\big) + \frac{1}{\epsilon}\int_{\rho'}^{1/5}\int_{\epsilon}^{\rho'}\frac{\dd{t}\dd{u}}{t u} F\big((\theta_{\epsilon} - t-u)/\epsilon\big)\\
&\
-\int_{\rho'}^{1/5}\int_{\epsilon}^{\rho'}\int_{\epsilon}^u \frac{\dd{t}\dd{u}\dd{v}}{t u v^2} f\big((\theta(t,u,v) - t-u-v)/v\big). \nonumber
\end{align}

We similarly let
\begin{align}\label{eq:defG68}
G_5 & = \frac{1}{\epsilon}\int_{\rho'}^{1/5}\int_{\rho'}^{t} \frac{\dd{t}\dd{u}}{t u} \,F\big((\theta_\epsilon - t-u)/\epsilon\big) \ + \ \frac1{\rho'}\int_{1/5}^{\tau_2}\int_{\rho'}^{t} \frac{\dd{t}\dd{u}}{t u} \,F\big((\theta(t) - t-u)/\rho'\big)\nonumber\\
&\ - \int_{\rho'}^{1/5}\int_{\rho'}^{t}\int_{\epsilon}^{\rho'} \frac{\dd{t}\dd{u}\dd{v}}{t u v^2} \,f\big((\theta(t,u,v) - t-u-v)/v\big),\nonumber\\
G_6 & = \frac{1}{\rho'}\int_{\tau_2}^{\tau_3}\int_{\rho'}^{\tau_1} \frac{\dd{t}\dd{u}}{t u} \,F\big((\theta(t) - t-u)/\rho'\big),
\end{align}
\begin{align*}
G_7 &= \frac1{\epsilon}\int_{\rho'}^{\tau_1}\int_{\rho'}^{t} \frac{\dd{t}\dd{u}}{t u} \,F\big((\theta_\epsilon - t-u)/\epsilon \big)\\
&\ - \int_{\rho'}^{\tau_1}\int_{\rho'}^{t} \int_{\epsilon}^{u} \frac{\dd{t}\dd{u}\dd{v}}{t u v^2} \,f\big((\theta(t,u,v) - t-u-v)/v\big),\\
G_8 &= \frac{1}{\epsilon}\int_{\tau_1}^{1/5}\int_{\rho'}^{\tau_1} \frac{\dd{t}\dd{u}}{t u} \,F\big((\theta_\epsilon - t-u)/\epsilon\big) \ + \int_{1/5}^{\tau_2}\int_{\rho'}^{\tau_1} \frac{\dd{t}\dd{u}}{t u^2} \,F\big((\theta(t) - t-u)/u\big)\\
&\ - \int_{\tau_1}^{1/5}\int_{\rho'}^{\tau_1} \int_{\epsilon}^{u} \frac{\dd{t}\dd{u}\dd{v}}{t u v^2} \,f\big((\theta(t,u,v) - t-u-v)/v\big).
\end{align*}

Recall the sieve functions $F,f$ satisfy $F(s)=2e^\gamma/s$ for $s\in[1,3]$, $f(s) = 2e^\gamma\log(s-1)/s$ for $s\in[2,4]$ and $F(s) = 2e^\gamma/s\cdot[1+\int_2^{s-1} f(t)\dd{t}]$ for all $s\ge1$.

The main bound is the following.

\begin{proposition}\label{prop:wu}
Let $0<\epsilon \le 0.1\le \rho' \le \tau_1 < 0.2 \le \tau_2 < \tau_3 \le \rho \le 0.3.$ Then for $I_n$, $G_n$, and $G(c)$ as in \eqref{eq:defIn},\eqref{eq:defG14},\eqref{eq:defG68}, and \eqref{eq:Gc}, we have
\begin{align}
S(\mathcal A, x^\rho) \  \lesssim \ \frac{\Pi(x)}{5e^\gamma} \bigg( \sum_{n=1}^8 G_n \ + \  G(\tfrac{1}{5})\sum_{n=9}^{21}I_n \bigg).
\end{align}
\end{proposition}
\begin{proof}
We first bound $S(\mathcal A,x^{c})$ for $c\in [\epsilon,1/5]$. By the Buchstab identity,
\begin{align*}
S(\mathcal A,x^{c}) = S(\mathcal A,x^{\epsilon}) - \sum_{x^{\epsilon}\le p< x^{c}}S(\mathcal A_p,p).
\end{align*}
Iterating twice more, we obtain
\begin{align}\label{eq:Buchstabiter}
S(\mathcal A,x^{c}) \ &= \ S(\mathcal A,x^{\epsilon}) - \sum_{x^{\epsilon}\le p_1< x^{c}} S(\mathcal A_{p_1},x^{\epsilon})\\
& + \sum_{x^{\epsilon}\le p_2< p_1< x^{c}} S(\mathcal A_{p_1p_2},x^{\epsilon}) - \sum_{x^{\epsilon}\le p_3< p_2< p_1< x^{c}} S(\mathcal A_{p_1p_2p_3},p_3).\nonumber
\end{align}
To each term $S(\mathcal A_{d},x^{\epsilon})$ above, we apply the linear sieve of level $x^{\theta_\epsilon}$ for $\theta_\epsilon = \frac{3-\epsilon}{5}$, as in Theorem \ref{thm:wellfactor}. And to each term $S(\mathcal A_{p_1p_2p_3},p_3)$, we apply the linear sieve of level $x^{\theta}$ for $\theta = \theta(t_1,t_2,t_3)$, where $p_i=x^{t_i}$. 

To handle the corresponding error terms, note for primes $x^{\epsilon}\le p_2<p_1<x^c\le x^{1/5}$ and $d\in\{1,p_1,p_1p_2\}$, the prime factors of $q/d$ above are bounded by $x^{\epsilon}$ so that the sets $\mathcal A_q$ are equidistributed to level $\theta_\epsilon$. Hence for each $x^{\epsilon}\le p_2<p_1< x^{1/5}$, by Proposition \ref{prop:lvlalpha1} with $u=\epsilon$, $b=1$, $p_1$, $p_1p_2$, we have
\begin{align*}
\sum_{\substack{q\le x^{\theta_{\epsilon}}\\q\mid P(x^{\epsilon})}} \widetilde{\lambda}^+(q) \, \Big(|\mathcal A_{q}| - \frac{|\mathcal A|}{\phi(q)}\Big) 
= \sum_{\substack{q\le x^{\theta_{\epsilon}}\\q\mid P(x^{\epsilon})}} \widetilde{\lambda}^+(q) \, \Big(\pi(x;q,-2) - \frac{\pi(x)}{\phi(q)}\Big) \ \ll_A \ \frac{x}{(\log x)^A},
\end{align*}
and 
\begin{align*}
\sum_{p_1}\sum_{\substack{q=p_1m\le x^{\theta_{\epsilon}}\\ m\mid P(x^{\epsilon})}} \widetilde{\lambda}^-(q) \, \Big(|\mathcal A_{q}| - \frac{|\mathcal A|}{\phi(q)}\Big) \ll_A \frac{x}{(\log x)^A},
\end{align*}
and
\begin{align*}
\sum_{p_2,p_1}\sum_{\substack{q=p_1p_2m\le x^{\theta_{\epsilon}}\\m\mid P(x^{\epsilon})}} \widetilde{\lambda}^+(q) \, \Big(|\mathcal A_{q}| - \frac{|\mathcal A|}{\phi(q)}\Big) \ll_A  \frac{x}{(\log x)^A}.
\end{align*}
In addition for each $p_3<p_2<p_1< x^{1/5}$, $p_i=x^{t_i}$, letting $\theta=\theta(t_1,t_2,t_3)$, by Proposition \ref{prop:lvlalpha1} with $b=p_1p_2p_3$,
\begin{align*}
\sum_{p_3,p_2,p_1}\sum_{\substack{q=p_1p_2p_3m\le x^\theta\\m\mid P(p_3)}} \widetilde{\lambda}^-(q) \, \Big(|\mathcal A_{q}| - \frac{|\mathcal A|}{\phi(q)}\Big) \ \ll_A \  \frac{x}{(\log x)^A}.
\end{align*}

Thus for $c\le \frac{1}{5}$, the linear sieve bounds give
\begin{align}\label{eq:iter0}
S(\mathcal A,x^{\epsilon}) \ \lesssim \ |\mathcal A|V(x^{\epsilon}) F\big(\tfrac{\theta_{\epsilon}}{\epsilon}\big),
\end{align}
and
\begin{align}\label{eq:iter1}
\sum_{x^{\epsilon}\le p_1< x^{c}} S(\mathcal A_{p_1},x^{\epsilon}) \ \gtrsim \ \sum_{x^{\epsilon}\le p_1< x^{c}}|\mathcal A|g(p_1) V(x^{\epsilon})f\big(\tfrac{\theta_{\epsilon}-t_1}{\epsilon}\big),
\end{align}
and
\begin{align}\label{eq:iter2}
\sum_{x^{\epsilon}\le p_2< p_1< x^{c}} S(\mathcal A_{p_1p_2},x^{\epsilon}) \ \lesssim \ \sum_{x^{\epsilon}\le p_2< p_1< x^{c}} |\mathcal A|g(p_1p_2) V(x^{\epsilon})F\big(\tfrac{\theta_{\epsilon}-t_1-t_2}{\epsilon}\big),
\end{align}
and
\begin{align}\label{eq:iter3}
 \sum_{x^{\epsilon}\le p_3< p_2< p_1< x^{c}} S(\mathcal A_{p_1p_2p_3},p_3) \ \gtrsim \
\sum_{x^{\epsilon}\le p_3< p_2< p_1< x^{c}}
|\mathcal A| g(p_1p_2p_3)  V(p_3)f\big(\tfrac{\theta-t_1-t_2-t_3}{t_3}\big).
\end{align}

Hence by \eqref{eq:iter0}, \eqref{eq:iter1}, \eqref{eq:iter2}, \eqref{eq:iter3}, we observe that \eqref{eq:Buchstabiter} becomes
\begin{align}
 &S(\mathcal A,x^{c}) \  \lesssim \ - \ |\mathcal A|\sum_{x^{\epsilon}\le p_3< p_2< p_1< x^c}
g(p_1p_2p_3)V(p_3)f\big(\tfrac{\theta-t_1-t_2-t_3}{\epsilon}\big)\\
& \ + \ |\mathcal A|V(x^{\epsilon})\bigg( F\big(\tfrac{\theta_{\epsilon}}{\epsilon}\big) - \sum_{x^{\epsilon}\le p_1< x^{c}}g(p_1)f\big(\tfrac{\theta_{\epsilon}-t_1}{\epsilon}\big) + \sum_{x^{\epsilon}\le p_2< p_1< x^{c}} g(p_1p_2)F\big(\tfrac{\theta_{\epsilon}-t_1-t_2}{\epsilon}\big)\bigg). \nonumber
\end{align}
Recall $V(z)\sim \mathfrak{S}_2/e^\gamma \log z$ by Mertens theorem.
Thus by partial summation and the prime number theorem, we obtain
\begin{align}
S(\mathcal A,x^{c}) \ \lesssim \ \frac{\Pi(x)}{e^\gamma}G(c),
\end{align}
for $G(c)$ as in \eqref{eq:Gc}. Hence for $c=\rho',\tau_1$ we have $c\in [\epsilon,1/5]$, so we bound $\Gamma_1$ as
\begin{align*}
\Gamma_1 = 4S(\mathcal A,x^{\rho'}) + S(\mathcal A,x^{\tau_1}) \ \lesssim \ \frac{\Pi(x)}{e^\gamma}\Big(4G(\rho') + G(\tau_1)\Big) = \frac{\Pi(x)}{e^\gamma}G_1.
\end{align*}

Now consider $c,c'\in[1/5,2/7]$. We shall apply the linear sieve of level $\theta(t_1)$, as in \eqref{eq:thetat}. In general, for $p_1=x^{t_1}$ and $\theta(t_1)$ as in \eqref{eq:thetat}, Proposition \ref{prop:lvlalpha1} gives
\begin{align*}
\sum_{p_1}\sum_{\substack{p_1\mid q,\;q\le x^{\theta(t_1)}\\q/p_1\mid P(p_1)}} \widetilde{\lambda}^-(q) \, \Big(|\mathcal A_{q}| - \frac{|\mathcal A|}{\phi(q)}\Big) \ \ll_A \ \frac{x}{(\log x)^A},
\end{align*}
so that for $c,c'\in[1/5,2/7]$, the linear sieve of level $\theta(t_1)$ gives
\begin{align}\label{eq:iter51}
\sum_{x^{c'}\le p_1< x^{c}}  S(\mathcal A_{p_1},p_1) \ \gtrsim \ \sum_{x^{c'}\le p_1< x^{c}} |\mathcal A|g(p_1) V(p_1)f\big(\tfrac{\theta(t_1)-t_1}{t_1}\big).
\end{align}
Thus by partial summation and the prime number theorem,
\begin{align*}
\Gamma_2 = \sum_{x^{\rho'} \le p < x^{\rho}} S(\mathcal A_{p}, x^{\rho'}) \lesssim
\frac{\Pi(x)}{e^\gamma}\big(G + \overline{G}(\rho)\big) = \frac{\Pi(x)}{e^\gamma}\,G_2.
\end{align*}

Similarly, we obtain
\begin{align}\label{eq:Gamma18}
\Gamma_n \ \lesssim \ \frac{\Pi(x)}{e^\gamma} G_n \qquad\qquad \text{for }1\le n\le 8.
\end{align}

Finally, for the remaining $\Gamma_n$, we apply the switching principle. Namely, for $\Gamma_9$ we have
\begin{align}
\Gamma_9 = \underset{\substack{x^{\tau_1} \le p_1 <p_2 <p_3 < x^{\tau_3}}}{\sum\sum\sum} S(\mathcal A_{p_1p_2p_3}, p_2) \ = \ S(\mathcal B, x^{1/2}) \ + \ O(x^{1/2})
\end{align}
for the set
\begin{align}\label{eq:Bswitch}
\mathcal B \;=\; \{p_1p_2p_3m-2 \le x \; : \;  x^{\tau_1} \le & p_1 <p_2 <p_3 < x^{\tau_3},
\ p'\mid m\Rightarrow p'\ge p_2\}. 
\end{align}
Note since $p_2,p_1 > x^{\tau_1} > x^{0.1}$, each $m$ above has at most 7 prime factors.

Now by a standard subdivision argument, $\mathcal B$ is similarly equidstributed in arithmetic progressions as is $\mathcal A$. Indeed, the basic idea is to partition
$\mathcal B = \bigcup_{r\le 7} \mathcal B^{(r)}$, where $\mathcal B^{(r)}$ is the subset corresponding to integers $m$ with $r$ prime factors. Then we cover the prime tuples $(p_1,\ldots,p_r)$ into hypercubes of the form
\begin{align*}
\big[\Delta^{l_1},\Delta^{l_1+1}\big)\times \cdots\times \big[\Delta^{l_r},\Delta^{l_r+1}\big),
\end{align*}
for $\Delta=1+(\log x)^{-B}$ with $B>0$ sufficiently large, and apply Corollary \ref{cor:lvlalphak} with $u=\epsilon$, $b=1$. This gives
\begin{align*}
\sum_{\substack{q\le x^{\theta_{\epsilon}}\\q\mid P(x^{\epsilon})}} \widetilde{\lambda}^+(q) \, \Big(|\mathcal B_{q}| - \frac{|\mathcal B|}{\phi(q)}\Big)
\ \ll_A \ \frac{x}{(\log x)^A}.
\end{align*}
Similarly for each $x^\eps<p_3'<p_2'<p_1'<x^{1/5}$, by Corollary \ref{cor:lvlalphak} with $u=\epsilon$ and $b=p_1'$, $p_1'p_2'$, we have
\begin{align*}
\sum_{p_1'}\sum_{\substack{q=p_1'm'\le x^{\theta_{\epsilon}}\\m'\mid P(x^{\epsilon})}} \widetilde{\lambda}^-(q) \, \Big(|\mathcal B_{q}| - \frac{|\mathcal B|}{\phi(q)} \Big) \ \ll_A \ \frac{x}{(\log x)^A},
\end{align*}
and 
\begin{align*}
\sum_{p_2',p_1'}\sum_{\substack{q=p_1'p_2'm'\le x^{\theta_{\epsilon}}\\m'\mid P(x^{\epsilon})}} \widetilde{\lambda}^+(q) \, \Big(|\mathcal B_{q}| - \frac{|\mathcal B|}{\phi(q)}\Big) \ \ll_A \  \frac{x}{(\log x)^A}.
\end{align*}
In addition for each $p_3'<p_2'<p_1'< x^{1/5}$, $p_i'=x^{t_i}$, letting $\theta=\theta(t_1,t_2,t_3)$, by Corollary \ref{cor:lvlalphak} with $b=p_1'p_2'p_3'$, we have
\begin{align*}
\sum_{p_3',p_2',p_1'}\sum_{\substack{q=p_1'p_2'p_3'm'\le x^\theta\\m'\mid P(p_3')}} \widetilde{\lambda}^-(q) \, \Big(|\mathcal B_{q}| - \frac{|\mathcal B|}{\phi(q)}\Big) \ \ll_A \   \frac{x}{(\log x)^A}.
\end{align*}

Iterating the Buchstab identity, we have
\begin{align*}
S(\mathcal B,x^{1/2}) \le S(\mathcal B,x^{1/5}) \ &= \ S(\mathcal B,x^{\epsilon}) - \sum_{x^{\epsilon}\le p_1'< x^{1/5}} S(\mathcal B_{p_1'},x^{\epsilon})\\
& + \sum_{x^{\epsilon}\le p_2'< p_1'< x^{1/5}} S(\mathcal B_{p_1'p_2'},x^{\epsilon}) - \sum_{x^{\epsilon}\le p_3'< p_2'< p_1'< x^{1/5}} S(\mathcal B_{p_1'p_2'p_3'},p_3'),
\end{align*}
and hence by the linear sieve bounds we obtain
\begin{align}
\Gamma_9 \lesssim S(\mathcal B, x^{1/2}) \le S(\mathcal B, x^{1/5}) \ &\lesssim \ e^{-\gamma}\,\frac{G(\tfrac{1}{5})}{\log x}\mathfrak{S}_2 |\mathcal B|.
\end{align}

Now to compute $|\mathcal B|$ in \eqref{eq:Bswitch}, by Lemma \ref{lem:Buchstab} we have
\begin{align*}
|\mathcal B| \ \sim \  \frac{x}{\log x}\int_{\tau_1 < t_1<t_2<t_3<\tau_2}\frac{\dd{t_1}\dd{t_2}\dd{t_3}}{t_1t_2^2t_3}\omega\Big(\frac{1-t_1-t_2-t_3}{t_2}\Big) = \frac{x}{\log x} \cdot I_9.
\end{align*}
Thus we have $\Gamma_9 \lesssim e^{-\gamma}G(\tfrac{1}{5}) \Pi(x)  I_9$. Similarly, we obtain
\begin{align}\label{eq:Gamma9}
\Gamma_n \ \lesssim \ G(\tfrac{1}{5})\frac{\Pi(x)}{e^{\gamma}} \,I_n \qquad\qquad (\text{for }9\le n\le 21).
\end{align}
Hence plugging \eqref{eq:Gamma18}, \eqref{eq:Gamma9} into Lemma \ref{lem:WuII42} completes the proof.
\end{proof}

Let
\begin{align}\label{eq:params}
\rho  &= 0.27195, & 
\tau_3 &= 0.24589, \nonumber\\
\rho' &= 0.12313, &
\tau_2 &= 0.20867,\\
\epsilon &= 0.002, & \tau_1 &= 0.16288. \nonumber
\end{align}

For such choices of parameters, we compute the following integrals from Proposition \ref{prop:wu},
\begin{align*}
\sum_{9\le n\le 21} I_n \ \le \ 0.174404,\qquad
\sum_{1\le n\le 8} G_n \ \le \ 28.34581,
\qquad  G(\tfrac{1}{5}) \ \le \ 5.99237.
\end{align*}
Thus by Proposition \ref{prop:wu}, we obtain the bound
\begin{align}\label{eq:pi2oneiteration}
\pi_2(x) 
\ & \lesssim \ S(\mathcal A, x^\rho) \  \lesssim \ \frac{\Pi(x)}{5e^\gamma} \bigg( \sum_{n=1}^8 G_n \ + \  G(\tfrac{1}{5})\sum_{n=9}^{21}I_n \bigg)
\ \lesssim  \ 3.30042\,\Pi(x).
\end{align}
We also record the individual bounds (see table at end of paper).

\subsection{Completing the proof of Theorem \ref{thm:twinbound}}
We shall refine our argument in certain cases for which Wu's iteration method \cite{WuII} applies directly without modification. As such we have chosen simplicity over full optimization.

In the lemma below, we consider the cases of sets $\mathcal A_{p_1p_2}$, where $p_1,p_2$ lie in a prescribed range, and such that for all multiples $bp_1p_2$, the sets $\mathcal A_{bp_1p_2}$ are equidistributed to level $x^\theta$ (we also need level $x^\theta$ for corresponding switched sets $\mathcal B$ of integers $mbp_1p_2-2$).

\begin{lemma}\label{lem:Wuiterate}
Let $\theta\in[1/2,1]$, $s\in[2,3]$, and $\mathcal A = \{p+2:p\le x\}$. There is a function $H_\theta(s)$, monotonically increasing in $\theta$ for fixed $s$ and decreasing in $s$ for fixed $\theta$, such that the following holds:
For each $(D_1,D_2)\in{\bf D}_2^{\textnormal{well}}(x^\theta)$, we have
\begin{align}\label{eq:Wusave}
\sum_{\substack{D_1<p_1<D_1^{1+\tau}\\D_2<p_2<D_2^{1+\tau}}}
S(\mathcal A_{p_1p_2}, z) \ \lesssim \ \frac{\Pi(x)}{e^\gamma}
\sum_{\substack{D_1<p_1<D_1^{1+\tau}\\D_2<p_2<D_2^{1+\tau}}} \frac{\log x}{\phi(p_1p_2)\log z}\big(F(s) - \frac{2e^\gamma}{s}H_\theta(s)\big),
\end{align}
where $z^s= x^\theta/p_1p_2$, provided \eqref{eq:programprimes} holds for $\lambda = \widetilde{\lambda}^+$ at level $D=x^\theta$ with $(x^{\eps_1},x^{\eps_2}) = (D_1,D_2)$, and provided for all vectors $(D_1,\ldots,D_r)\in{\bf D}_r^{\textnormal{well}}(x^\theta)$ extending $(D_1,D_2)$,
\begin{align*}
\sum_{\substack{b=p_1\cdots p_r\\ D_i< p_i \le D_i^{1+\tau}}}
\sum_{\substack{b\mid d,\;d\le x^\theta\\(d,a)=1}} &\widetilde{\lambda}^\pm(d) \,\Big(\pi(x;d,a)-\frac{\pi(x)}{\phi(d)}\Big) \ \ll_{a,\eps,A}
\frac{x}{(\log x)^A}.
\end{align*}
\end{lemma}
\begin{proof}
Wu iterates the weighted sieve inequality (Lemmas 4.1 and 4.2 \cite{WuII}), on the subset of (non-switched) terms whose sieving parameter $s$ lies in the interval $s\in[2,3]$. This yields a recurrence relation for a function $H_\theta(s)$, which encodes the percent savings over the (normalized) linear sieve $sF(s)/(2e^\gamma)$.

Starting from a term $S(\mathcal A_{p_1p_2}, z)$, each successive iteration of the weighted sieve inequality is composed of terms of the form $S(\mathcal A_{bp_1p_2}, z')$ for some multiple $bp_1p_2$ corresponding to some vector extending $(D_1,D_2)$. By assumption, all such sets are equidistributed to level $x^\theta$, when weighted by the upper/lower linear sieve. Similarly, the switched sets are also equidistributed to level $x^\theta$ (here we only need the upper bound weights $\widetilde{\lambda}^+$). Finally, the savings function $H_\theta(s)$ inherits the stated monotonicity properties by construction of the iteration.
\end{proof}
The function $H_\theta$ depends on the known level of distribution $x^\theta$ (i.e. Wu used $\theta=\frac{1}{2}$ for Goldbach, and $\theta=\frac{4}{7}$ for twin primes). For parameters as in Tables 1 and 2 \cite[pp.30--32]{WuII},
\[H_{1/2}(t) \ge \begin{cases}
0.0223939 \qquad \text{if} \ 2.0 <  t \le \, 2.2, \\
0.0217196 \qquad \text{if} \ 2.2 < t \le \, 2.3,\\
0.0202876 \qquad \text{if} \ 2.3 < t \le \, 2.4,\\
0.0181433 \qquad \text{if} \ 2.4 < t \le \, 2.5,\\
0.0158644 \qquad \text{if} \ 2.5 < t \le \,2.6, \\
0.0129923 \qquad \text{if} \ 2.6 < t \le  \,2.7, \\
0.0100686 \qquad \text{if} \ 2.7 < t \le \, 2.8, \\
0.0078162 \qquad \text{if} \ 2.8 < t \le \, 2.9, \\
0.0072943 \qquad \text{if} \ 2.9 < t \le \, 3.0, \\
0  \qquad\qquad\qquad\textnormal{else,}
\end{cases}\]
and
\[H_{4/7}(t) \ge \begin{cases}
0.0287118 \qquad \text{if} \ 2.0 \, \le \, t \le \, 2.1,\\
0.0280509 \qquad \text{if} \ 2.1 <  t \le \, 2.2, \\
0.0264697 \qquad \text{if} \ 2.2 < t \le \, 2.3,\\
0.0241936 \qquad \text{if} \ 2.3 < t \le \, 2.4,\\
0.0214619 \qquad \text{if} \ 2.4 < t \le \, 2.5, \\
0.0183875 \qquad \text{if} \ 2.5 < t \le \,2.6, \\
0.0149960 \qquad \text{if} \ 2.6 < t \le  \,2.7, \\
0.0117724 \qquad \text{if} \ 2.7 < t \le \, 2.8, \\
0.0094724 \qquad \text{if} \ 2.8 < t \le \, 2.9, \\
0.0090024 \qquad \text{if} \ 2.9 < t \le  \,3.0, \\
0  \qquad\qquad\qquad\textnormal{else.}
\end{cases}
\]
As such, in \cite[Theorem 3]{WuII} Wu obtained $\pi_2(x)/\Pi(x) \lesssim (7/2)(1-H_{4/7}(2.1))\le 3.39951$.

To complete our proof of Theorem \ref{thm:twinbound} we apply Lemma \ref{lem:Wuiterate}, now valid up to level $x^{\frac{7}{12}}$ by Corollary \ref{cor:May712} and Proposition \ref{prop:JMprogramprimes} for $\widetilde{\lambda}^+$. Note when the largest integration variable is $t\ge \frac{1}{5}$, we have $\theta(t)=\frac{2-t}{3}\le\, \frac{7}{12}$ iff $t \ge \frac{1}{4}$. A key feature we use to satisfy the conditions of Lemma \ref{lem:Wuiterate} is that the level $x^{\theta(t_1)}$ persists, since the largest prime $p_1=x^{t_1}$ is preserved through successive iterations.

Thus in practice, Lemma \ref{lem:Wuiterate} simply amounts to modifying the integral in $G_2$ by substituting $F(s) - \frac{2e^\gamma}{s}H_\theta(s)$ in for $F(s)$, $s = (\theta(t) - t-u)/u$, when $t\ge \frac{1}{4}$ (the only parameter $\ge \frac{1}{4}$ is $\rho$, so we only refine $G_2$). Denote this as $G^{\textnormal{Wu}}_2$. For ease we also use $H_\theta(s)\ge H_{4/7}(s)$, by monotonicity in $\theta$. Doing so, with the same parameter choices \eqref{eq:params}, we obtain $ G_2^{\textnormal{Wu}} \le\, -5.598667$ and hence
\begin{align}
\pi_2(x) 
\ & \lesssim \ \frac{\Pi(x)}{5e^\gamma} \bigg( G_2^{\textnormal{Wu}}+ \sum_{1\le n\neq2 \le 8} G_n \ + \  G(\tfrac{1}{5})\sum_{9\le\, n\le\, 21} I_n \bigg)
\ \lesssim  \ 3.299552\,\Pi(x).
\end{align}

This completes the proof of Theorem \ref{thm:twinbound}.

For slight numerical gains, one may compute $H_\theta(s)$ when $\theta\in [\frac{4}{7},\frac{7}{12}]$, by tweaking the formulae in \cite{WuII}. More substantially, Wu defined a lower bound savings $h_\theta(s)$, for a substitution of $f(s)$ by $f(s) + \frac{2e^\gamma}{s}h_\theta(s)$. But in practice, to compute $h$ would require derivations (analogous to $H$) of as yet undetermined formulae. We leave these to the reader.

\vspace{3mm}
\begin{center}
\begin{tabular}{r|r|rlrl}
$n$ & $G_n$  & $n$ & $I_n$ & $n$ & $I_n$\\
\hline
1 & 39.00163  & 9  & 0.0332157 & 17 & 0.000315\\
2&$-$5.591009 & 10 & 0.0228322 & 18 & 0.000269\\
3&$-$3.986553 & 11 & 0.0092564 & 19 & 0.000164\\
4&$-$5.060499 & 12 & 0.0150101 & 20 & $\le \ 2.70\cdot10^{-6}$\\
5 &  1.864133 & 13 & 0.0547244 & 21 & $\le \ 5.50\cdot10^{-9}$\\
6 &  0.741181 & 14 & 0.0260202 &    & \\
7 &  0.453663 & 15 & 0.0124636 &    & \\
8 &  0.923736 & 16 & 0.0001314 &    & 
\end{tabular}
\end{center}

\section*{Acknowledgments}
The author is grateful to James Maynard for suggesting the problem and for many valuable discussions. The author also thanks Carl Pomerance, Kyle Pratt, and the anonymous referee for helpful feedback. The author was supported by a Clarendon Scholarship at the University of Oxford.

\bibliographystyle{amsplain}

\end{document}